\documentclass[a4paper,11pt,final]{article}

\usepackage{graphicx}
\usepackage{amssymb,amsmath,amsthm,mathrsfs,xfrac}
\usepackage{enumitem}
\usepackage{a4wide}
\usepackage{hyperref}
\usepackage{array}

\numberwithin{equation}{section}
\allowdisplaybreaks[4]

\newtheorem{proposition}{Proposition}[section]
\newtheorem{theorem}[proposition]{Theorem}
\newtheorem{notation}[proposition]{Notation}
\newtheorem{lemma}[proposition]{Lemma}
\newtheorem{definition}[proposition]{Definition}
\newtheorem{corollary}[proposition]{Corollary}
\newtheorem{remark}[proposition]{Remark}

\renewenvironment{proof}{\smallskip\noindent\emph{\textbf{Proof.}}%
  \hspace{1pt}}{\hspace{-5pt}{\nobreak\quad\nobreak\hfill\nobreak%
    $\square$\vspace{2pt}\par}\smallskip\goodbreak}

\newenvironment{proofof}[1]{\smallskip\noindent{\textbf{Proof~of~#1.}}%
  \hspace{1pt}}{\hspace{-5pt}{\nobreak\quad\nobreak\hfill\nobreak%
    $\square$\vspace{2pt}\par}\smallskip\goodbreak}

\newcommand{\C}[1]{\mathbf{C}^{#1}}
\newcommand{\AC}{\mathbf{AC}}
\newcommand{\Cc}[1]{\mathbf{C}_c^{#1}}

\newcommand{\BV}{\mathbf{BV}}
\newcommand{\SBV}{\mathbf{SBV}}
\newcommand{\SBVloc}{\mathbf{SBV}_{\mathbf{loc}}}
\newcommand{\BVloc}{\mathbf{BV}_{\mathbf{loc}}}

\newcommand{\Lloc}[1]{{\mathbf{L}_{\mathbf{loc}}^{#1}}}
\newcommand{\W}[2]{{\mathbf{W}^{#1,#2}}}
\newcommand{\Wloc}[2]{{\mathbf{W}_{\mathbf{loc}}^{#1,#2}}}
\newcommand{\modulo}[1]{{\left|#1\right|}}
\newcommand{\norma}[1]{{\left\|#1\right\|}}
\newcommand{\caratt}[1]{{\chi_{\strut#1}}}
\newcommand{\reali}{{\mathbb{R}}}
\newcommand{\naturali}{{\mathbb{N}}}

\newcommand{\sgn}{\mathop{\rm sgn}}
\newcommand{\card}{\mathop{\rm card}}

\renewcommand{\L}[1]{{\mathbf{L}^#1}}

\renewcommand{\epsilon}{\varepsilon}
\renewcommand{\phi}{\varphi}
\renewcommand{\theta}{\vartheta}

\renewcommand{\d}[1]{\mathinner{\mathrm{d}{#1}}}

\newcommand{\Leb}{\mathinner\mathfrak{Leb}}

\renewcommand{\SS}{S^{{\scriptscriptstyle HJ}}}
\newcommand{\II}{I^{\scriptscriptstyle HJ}}
\renewcommand{\ss}{S^{\scriptscriptstyle CL}}
\newcommand{\ii}{I^{\scriptscriptstyle CL}}

\makeatletter
\let\@fnsymbol\@arabic
\makeatother

\title{Initial Data Identification in\\Conservation Laws and Hamilton--Jacobi Equations}

\author{Rinaldo M.~Colombo$^1$ \and Vincent Perrollaz$^2$}

\begin{document}

\maketitle

\footnotetext[1]{INdAM Unit, University of Brescia,
  Italy. \texttt{rinaldo.colombo@unibs.it}}

\footnotetext[2]{Institut Denis Poisson, Universit\'e de Tours, CNRS
  UMR 7013, Universit\'e d'Orl\'eans,\\
  \texttt{vincent.perrollaz@lmpt.univ-tours.fr}}

\begin{abstract}

  \noindent In the scalar 1D case, conservation laws and
  Hamilton--Jacobi equations are deeply related. For both, we
  characterize those profiles that can be attained as solutions at a
  given positive time corresponding to at least one initial
  datum. Then, for each of the two equations, we precisely identify
  all those initial data yielding a solution that coincide with a
  given profile at that positive time. Various topological and
  geometrical properties of the set of these initial data are then
  proved.

  \medskip

  \noindent\textit{2000~Mathematics Subject Classification:} 35L65, 35F21,
  93B30, 35R30.

  \medskip


\end{abstract}

% \tableofcontents

\section{Introduction}
\label{sec:Intro}

Under suitable conditions on the flow $f \colon \reali \to \reali$ and
on the initial datum, solutions to a scalar conservation law in $1$
space dimension, namely to
\begin{equation}
  \label{eq:1}
  \partial_t u + \partial_x f (u) = 0 \,,
\end{equation}
are known to be obtained through $u = \partial_x U$ from solutions to
the Hamilton--Jacobi equation
\begin{equation}
  \label{eq:35}
  \partial_t U + f (\partial_x U) = 0 \,.
\end{equation}
A peculiar feature of these equations is their irreversibility. In
particular, in the case of~\eqref{eq:1}, inexorable shock formations
cause an unavoidable loss of information, so that different initial
data may well evolve into the same profile. Usual identification
techniques, often based on linearizations or fixed point arguments,
have no chances to be effective when dealing with~\eqref{eq:1}
or~\eqref{eq:35}.

Below, we provide a full characterization of the set of the initial
data for~\eqref{eq:1}, respectively~\eqref{eq:35}, that evolve into a
given profile. Geometric and topological properties of this set are
also obtained. To this aim, a refinement of the results
in~\cite{KarlsenRisebro2002}, see also~\cite{CorriasFalconeNtalini,
  Kruzkov1964}, on the relation between~\eqref{eq:1} and~\eqref{eq:35}
had to be obtained.

For any suitable initial datum $u_o$, we denote by
$(t,x) \to \ss_t u_o (x)$ the weak entropy solution
to~(\ref{eq:1}). Symmetrically, we denote by $(t,x) \to \SS_t U_o (x)$
the viscosity solution to~\eqref{eq:35}. Below, we consider the case
of a uniformly convex $\C2$ flux $f$ and we obtain complete
characterizations of both sets
\begin{equation}
  \label{eq:2}
  \begin{array}{rcl}
    \ii_T (w)
    & :=
    & \left\{
      u_o \in \L\infty (\reali; \reali) \colon \ss_T u_o = w
      \right\}
      \qquad \mbox{ and }
    \\[3pt]
    \II_T (W)
    & :=
    & \left\{
      U_o \in \W{1}{\infty} (\reali; \reali) \colon \SS_T U_o = W
      \right\}
  \end{array}
\end{equation}
and we use the notation $I_T$ whenever we refer to both sets
in~\eqref{eq:2}.

First, we identify those profiles such that the corresponding set
$I_T$ is non empty. This proof is constructive, in the sense that an
initial datum in $I_T$ is explicitly constructed, see
Theorem~\ref{thm:extremal}. Here, we consider in detail the case of
the conservation law~\eqref{eq:1}. A key role is played by the decay
of rarefaction waves, a phenomenon typically described through Oleinik
decay estimates that goes back to~\cite{Oleinik1957Originale}, was
recently improved in~\cite{Glass}, extended to systems of conservation
laws in~\cite{BressanDecay} and of balance laws
in~\cite{CleopatraKonstantina}, see also the reference
texts~\cite[Chapter~6, Ex.5]{BressanLectureNotes}
and~\cite[Theorem~11.2.1]{DafermosBook}. For related problems dealing
with the reachable set of~\eqref{eq:1}, also in the case of the
initial -- boundary value problem, we refer to~\cite{AnconaMarson98,
  Horsin1998} and~\cite{AdimurthiGhoshalGowda2014}.

Once $I_T$ is ensured to be non empty, in its characterization as well
as in establishing its properties a key role is played by two sets,
say $X_i$ and $X_{ii}$, whose precise definitions are
in~\eqref{eq:22}. For $x$ varying in the former one, $X_i$, the value
attained at $x$ by any initial datum in $I_T$ is essentially uniquely
determined. On the contrary, for $x$ varying in the latter one,
$X_{ii}$, the value attained at $x$ by any initial datum in $I_T$ is
subject to rather loose constraints. Moreover, coherently with the
finite propagation speed typical of~\eqref{eq:1} and~\eqref{eq:35},
the values attained at $x$ by any initial datum in $I_T$ on each of
the different connected components of $X_i$ and $X_{ii}$ are entirely
independent from each other.

Instrumental in these proofs is the ability to go back and forth
between solutions to~\eqref{eq:1} and solutions to~\eqref{eq:35}. To
this aim, we needed to complete the results
in~\cite{KarlsenRisebro2002} that deal with the connection
from~\eqref{eq:35} to~\eqref{eq:1}. Indeed, Proposition~\ref{prop:sol}
details how to pass from solutions to~\eqref{eq:1} to solutions to~\eqref{eq:35}.

On the basis of the obtained characterizations, several properties of
$\ii_T (w)$ are then proved. First, we re-obtain its convexity, which
was already stated in~\cite{GosseZuazua}. Then, the unique extreme
point of $\ii_T (w)$ is fully characterized and we prove that,
remarkably, this set is a cone admitting no finite dimensional
extremal faces.

The characterization below directly shows that as soon as $\ii_T$ is
non empty, then also $\ii_T \cap \BV (\reali; \reali)$ is non empty,
meaning that any profile reached by an initial datum with unbounded
total variation can also be reached by a (different) initial datum in
$\BV$. Moreover, we prove that $\ii_t$ always contains one sided
Lipschitz continuous functions but more regular initial data may also
be available. The initial datum constructed \emph{``prolonging
  backwards all shocks''} yields a solution whose interaction
potential~\cite[Formula~(10.10)]{BressanLectureNotes} is
\emph{constant} on $\left]0, T\right[$.

\smallskip

Further motivations for the present study are provided by parameter
identification or inverse problems based on~\eqref{eq:1}
or~\eqref{eq:35}. In particular, we defer to the related
paper~\cite{GosseZuazua} that motivates the present problem through
applications to the study of sonic booms in~\eqref{eq:1}, also
providing several illustrative examples and visualizations. In the
case of~\eqref{eq:35}, $U$ is typically the value function associated
to a time reversed control problem, $f$ being related to the dynamics
and to the running cost, with $U_o$ playing the role of the terminal
cost. Here, the present result amounts to characterizing the terminal
cost corresponding to given initial cost,
see~\cite[Section~10.3]{Evans} for further connections to optimal
control problems. The present analytic results can also help in
numerical investigations such as those in~\cite{AllahverdiPozoZuazua,
  CastroPalaciosZuazua2008, LecarosZuazua2014, LecarosZuazua2016}.

\smallskip

Sections~\ref{subs:Notation} to~\ref{subs:Geo} collect the analytic
results, while all proofs are deferred to
sections~\ref{sec:proof:Notation} to~\ref{sec:proof:Geo}

\section{Notation and Preliminary Results}
\label{subs:Notation}

Throughout, $T$ is fixed and strictly positive. Below, we mostly refer
to~\cite[\S~3.2]{AmbrosioFuscoPallara} for results about $\BV$
functions. In Section~\ref{sec:proof:Notation} we briefly recall the
definition and the min properties of $\SBV (\reali; \reali)$, refer
to~\cite{AmbrosioDeLellis2004} or~\cite[\S~1.7]{DafermosBook} for more
details. As usual, we also use functions $u$ in
$\BVloc (\reali; \reali)$, respectively in $\SBVloc (\reali; \reali)$,
meaning that the restriction $u_{\vert I}$ of $u$ to any bounded real
interval $I$ is in $\BV (I; \reali)$, respectively in
$\SBV (I; \reali)$.  If $u \in \BVloc (\reali; \reali)$, then we set
$u (x\pm) = \lim_{\xi\to x\pm} u (\xi)$ and we convene that we choose
as $u$ the left continuous representative of its class, so that
$u (x) = u (x-)$.

We assume the following condition on the function
defining~\eqref{eq:1} or~\eqref{eq:35}, where $c$ is a suitable
positive constant:
\begin{equation}
  \label{eq:H}
  \text{\textbf{(F):} \qquad$f \in \C2 (\reali; \reali)$ is such that
    $f'' \geq c >0 $ and
    $f (0) = \min_{\reali} f = 0$.}
\end{equation}
Clearly, the latter part of the above condition is not restrictive,
since it can be achieved through \emph{ad hoc} translations of the $u$
or $\partial_x U$ variable and of the flux $f$.

As general references on the theory of scalar conservation laws we
use~\cite[Chapter~6]{BressanLectureNotes},
\cite[Chapter~XI]{DafermosBook} or~\cite[Vol.~1,
Chapter~2]{SerreBooks}.  As specified in Definition~\ref{def:HCL}
below, by \emph{solution} to~\eqref{eq:1}, we always mean a weak
entropy solution in the sense of~\cite[\S~4.4]{BressanLectureNotes},
see also~\cite[Definition~6.2.1]{DafermosBook}
or~\cite[Definition~2.3.3]{SerreBooks}.

\begin{definition}
  \label{def:HCL}
  A map
  $u \in \L\infty ([0,T] \times \reali; \reali) \cap \C0 ([0,T];
  \Lloc1 (\reali; \reali))$ is a \emph{weak entropy solution}
  to~\eqref{eq:1} if
  \begin{displaymath}
    \iint_{\reali^2}
    \left(
      \modulo{u - k} \, \partial_t \phi
      +
      \sgn (u-k) \, \left(f (u) - f (k)\right) \, \partial_x \phi
    \right)
    \d{t} \d{x} \geq 0
    \qquad
    \begin{array}{@{}r@{\,}c@{\,}l@{}}
      \forall \phi
      & \in
      & \Cc1 (\reali^+ \times \reali; \reali^+) \,,
      \\
      \forall k
      & \in
      & \reali \,.
    \end{array}
  \end{displaymath}
\end{definition}

\noindent Consider the Cauchy problem for~\eqref{eq:1} with an
$\L\infty$ initial datum assigned at time $t=0$. Then,
by~\cite[Theorem~16.1]{Smoller}, condition~\eqref{eq:H} ensures that
as soon as a weak entropy solution exists, then it has locally bounded
total variation in space at any positive time.

Concerning~\eqref{eq:35}, we use the standard definition of viscosity
solution based on super-solutions and sub-solutions,
see~\cite[Chapter~10]{Evans}, \cite[Definition~I.1]{CrandallLions1983}
or~\cite[Section~1]{KarlsenRisebro2002}.
\begin{definition}
  \label{def:HJ}
  A map
  $U \in \W{1}{\infty} ([0,T] \times \reali; \reali) \cap \C0 ([0,T];
  \Wloc{1}{\infty} (\reali; \reali))$ is a \emph{viscosity solution}
  to~\eqref{eq:35} if for all
  $\phi \in \C\infty (\left]0,T\right[ \times \reali; \reali)$ and all
  $(t_o, x_o) \in \left]0, T\right[ \times \reali$
  \begin{itemize}
  \item if $U-\phi$ has a local maximum at $(t_o, x_o)$, then
    $\partial_t \phi(t_o, x_o) + f\left(\partial_x \phi (t_o,
      x_o)\right) \leq 0$;
  \item if $U-\phi$ has a local minimum at $(t_o, x_o)$, then
    $\partial_t \phi(t_o, x_o) + f\left(\partial_x \phi (t_o,
      x_o)\right) \geq 0$.
  \end{itemize}
\end{definition}
\noindent For the existence of a semigroup generated by~\eqref{eq:35}
yielding solutions in the sense of Definition~\ref{def:HJ}, we refer
for instance to the classical
result~\cite[Theorem~VI.2]{CrandallLions1983}.

The space derivation, i.e., the map $U \to u = \partial_x U$, shows
the equivalence between solutions to~\eqref{eq:35} in the sense of
Definition~\ref{def:HJ} and solutions to~\eqref{eq:1} in the sense of
Definition~\ref{def:HCL}, see~\cite[Theorem~1.1]{KarlsenRisebro2002}
and the references therein.

\smallskip

Throughout this paper, the following function plays a key role.

\begin{notation}
  \label{not:p}
  For a fixed $w \in \L\infty (\reali;\reali)$, we denote
  \begin{equation}
    \label{eq:15}
    \begin{array}{ccccl}
      p
      & \colon
      & \reali
      & \to
      & \reali
      \\
      &
      & x
      & \mapsto
      & x - T \, f' \! \left(w (x)\right) \,.
    \end{array}
  \end{equation}
\end{notation}

\noindent In the case of~\eqref{eq:1}, as soon as
$\ii_T (w) \neq \emptyset$, $p$ assigns to each $x \in \reali$ the
intersection of the minimal backward characteristic for~\eqref{eq:1}
through $(T,x)$, see~\cite[Chapter~X]{DafermosBook}, with the axis
$t=0$. In the case of~\eqref{eq:35}, we clearly set
$p (x) := x - T \, f'\left(\partial_x W (x)\right)$.

The choice of the \emph{left} continuous representative of $w$ is here
crucial to obtain \emph{minimal} backward characteristics.

Oleinik condition on the decay of positive waves~\cite{BressanDecay,
  Glass, Oleinik1957Originale}, see also~\cite[Chapter~6,
Ex.5]{BressanLectureNotes} and~\cite[Theorem~11.2.1]{DafermosBook}, is
equivalent to require that $p$, defined in~\eqref{eq:15}, be weakly
increasing:
\begin{equation}
  \label{eq:O}
  \mbox{\textbf{(O):}} \quad
  \mbox{for all } x \in \reali \mbox{ and }y \in \reali^+ \setminus \{0\}
  \quad
  \begin{array}{l}
    p(x) \leq p (x+y)\,,\mbox{ equivalently}
    \\
    f'\left(w (x+y)\right) - f' \left(w(x)\right) \leq y/T \,.
  \end{array}
\end{equation}

\noindent On the basis of Oleinik Condition~\eqref{eq:O}, we partition
$\reali$ into two sets $X_i$ and $X_{ii}$ that play a key role in the
sequel.

\begin{proposition}
  \label{prop:D}
  Let~\eqref{eq:H} hold and $T$ be positive. Fix
  $w \in \L\infty (\reali; \reali)$ such that~\eqref{eq:O} holds and
  $p \in \SBVloc (\reali; \reali)$. Introduce the sets
  \begin{equation}
    \label{eq:22}
    \begin{array}{rcl}
      X_i
      & :=
      & p \left(\!
        \left\{x\in \reali \colon
        p \mbox{ is differentiable at }x \mbox{ and }
        p' (x) \neq 0
        \right\}
        \right) \,,
      \\
      X_{ii}
      & :=
      & \!\bigcup_{x \in \reali} \left]p (x-), p (x+)\right[ \,.
    \end{array}
  \end{equation}
  Then, $\reali \setminus \left(X_i \cup X_{ii}\right)$ has Lebesgue
  measure $0$.
\end{proposition}

We now investigate the equivalence between the Conservation
Law~\eqref{eq:1} and the Hamilton--Jacobi equation~\eqref{eq:35}. A
key result is~\cite[Theorem~1.1]{KarlsenRisebro2002}, to which we
provide here a completion, in the sense explained through the
following diagrams:
\begin{displaymath}
  \begin{array}{c@{\qquad\qquad\qquad}c}
    \mbox{\cite[Theorem~1.1]{KarlsenRisebro2002}\ }
    &\mbox{Proposition~\ref{prop:sol}\qquad\qquad\qquad}
    \\
    \begin{array}{c@{}ccc@{}c}
      & U_o
      & \longrightarrow
      & \SS_t U_o
      \\
      \partial_x
      & \big\downarrow
      &
      & \big\downarrow
      & \partial_x
      \\
      & u_o
      & \longrightarrow
      & \ss_t u_o
    \end{array}
      & \begin{array}{c@{}ccc@{}c}
          & U_o
          & \longrightarrow
          & \SS_t U_o
          \\
          \int^x
          & \big\uparrow
          &
          & \big\uparrow
          & \mbox{Formula~\eqref{eq:31}}
          \\
          & u_o
          & \longrightarrow
          & \ss_t u_o
        \end{array}
  \end{array}
\end{displaymath}

\begin{proposition}
  \label{prop:sol}
  Let $f$ satisfy~\eqref{eq:H}. Fix
  $u_o \in \L\infty (\reali; \reali)$ and call $u$ the solution in the
  sense of Definition~\ref{def:HCL} to the Cauchy problem for the
  Conservation Law~\eqref{eq:1} with datum $u_o$ at time $t=0$. For a
  path $\gamma \in \W1\infty ([0,T]; \reali)$ and a constant
  $c \in \reali$, define
  \begin{equation}
    \label{eq:31}
    U (t,x)
    =
    \int_{\gamma (t)}^x u (t,\xi) \d\xi
    +
    \int_0^t
    \left(
      \dot\gamma (\tau) \, u\left(\tau, \gamma (\tau)\right)
      -
      f\left(u\left(\tau, \gamma (\tau)\right)\right)
    \right)\d\tau
    +
    c\,.
  \end{equation}
  Then, $U$ solves Hamilton--Jacobi equation~\eqref{eq:35} with datum
  $U_o (x) = \int_{\gamma (0)}^x u_o (\xi)\d\xi + c$ in the sense of
  Definition~\ref{def:HJ}.
\end{proposition}

\section{Construction of a Remarkable Element of
  $\ii_t (w)$}
\label{subs:NonEmpty}

We now prove that Oleinik Condition~\eqref{eq:O} characterizes those
profiles $w$ such that $\ii_T (w) \neq \emptyset$. Indeed, if a
profile $w$ satisfies Oleinik condition~\eqref{eq:O}, then the
conservation law~\eqref{eq:1} can be integrated backwards in time,
taking $w$ as final datum at time $T$ and yielding a $\BV$ initial
profile at time $0$. Technically, we reverse the space variable,
rather than reversing time, and we explicitly construct an element of
$\ii_T (w)$ that will play a key role in the sequel.

\begin{theorem}
  \label{thm:extremal}
  Let~\eqref{eq:H} hold and $T$ be positive. Fix
  $w \in \L\infty (\reali; \reali)$ such that~\eqref{eq:O}
  holds. Then, there exists a unique function
  $u_o^* \in \L\infty(\reali; \reali)$ characterized by each one of
  the following two equivalent conditions:
  \begin{enumerate}[label={\sl(\roman*\/$^*$)}]
  \item \label{thm:ex:1} $u_o^* (x) = \tilde u (T,-x)$, where
    $\tilde u$ is a solution to $ \left\{
      \begin{array}{l}
        \partial_t \tilde u + \partial_x f (\tilde u) = 0 \,,
        \\
        \tilde u (0,x) = w (-x) \,.
      \end{array}
    \right.$

  \item \label{thm:ex:3} $u_o^*$ is such that
    \begin{enumerate}[label={\sl(ii$^*$.\roman*\/)}]
    \item \label{thm:ex:i} for all $x \in \reali$ where $p$ is
      differentiable and $p' (x) \neq 0$,
      \begin{displaymath}
        \lim_{y\to x} \frac{1}{p (y) - p (x)} \int_{p (x)}^{p (y)} u_o^* (\xi) \d\xi
        =
        w (x) \,;
      \end{displaymath}
    \item \label{thm:ex:ii} for all $x \in \reali$ where
      $w (x-) \neq w (x+)$, for all $v \in [w(x+),w(x-)]$,
      \begin{displaymath}
        \begin{array}{@{}r@{\,}c@{\,}l@{}}
          \displaystyle
          \frac{1}{T} \int_{x-T f'(v)}^{x-T f'(w(x+))} \!\! u_o^*(\xi) \d{\xi}
          & =
          & \left(v \, f'(v)-f(v)\right)
            -
            \left(
            w(x+) \, f'\left(w(x+)\right)-f\left(w(x+)\right)
            \right) ,
          \\[10pt]
          \displaystyle
          \frac{1}{T} \int_{x-T f'(w(x-))}^{x-T f'(v)} \!\! u_o^*(\xi) \d{\xi}
          & =
          & \left(w(x-) \, f'\left(w(x-)\right)-f\left(w(x-)\right)\right)
            -
            \left(v \, f'(v)-f(v)\right) .
        \end{array}
      \end{displaymath}
    \end{enumerate}
  \end{enumerate}
  \noindent Moreover, $u_o^*$ enjoys the following properties:
  \begin{enumerate}[label=\textsl{(\arabic*$^*$)}]
  \item \label{thm:exi:1} $u_o^* \in \ii_T (w)$;
  \item \label{thm:exi:2} $u_o^*$ is one sided Lipschitz and
    $u_o^* \in \BV (\reali; \reali)$;
  \item \label{thm:exi:3} for all $x \in \reali$ and for all
    $y \in \reali^+$,
    $f'\left(u_o^* (x-y)\right) - f'\left(u_o^* (x)\right) \leq
    \dfrac{y}{T}$.
  \end{enumerate}
\end{theorem}

We underline that condition~\ref{thm:ex:3} naturally determines two
subsets of $\reali$, related to $X_i$ and $X_{ii}$. Indeed,
in~\ref{thm:ex:3} we explicitly specify the exact sets of those points
$x$ where the two conditions~\ref{thm:ex:i} and~\ref{thm:ex:ii} have
to be satisfied by $w$ at time $t = T$. Essentially, the results above
show that if $u_o \in \ii_T (w)$, then the restriction
${u_o}_{\vert X_i}$ yields the continuous part of $w$, while
${u_o}_{\vert X_{ii}}$ yields the shocks.

As a first consequence of Theorem~\ref{thm:extremal} we obtain the
following characterization of those profiles $w$ such that $\ii_T (w)$
is non empty.

\begin{corollary}
  \label{thm:1}
  Let~\eqref{eq:H} hold and $T$ be positive. Fix
  $w \in \L\infty (\reali; \reali)$. With the notations~\eqref{eq:2}
  and~\eqref{eq:15}, the following statements are equivalent:
  \begin{enumerate}[label={\sl(\alph*)}]
  \item \label{thm:1:1} $\ii_T (w) \neq \emptyset$;
  \item \label{thm:1:2} a suitable representative of $w$ satisfies
    Oleinik Condition~\eqref{eq:O}.
  \end{enumerate}
\end{corollary}

Note that condition~\ref{thm:1:2}, and hence the requirement
$\ii_T(w) \neq \emptyset$, also ensures that
$w \in \BVloc (\reali; \reali)$. Hence, $w$ admits a left continuous
representative satisfying~\ref{thm:1:2} for all $x \in \reali$.

\smallskip

The translation of the results above to the case of Hamilton--Jacobi
equation~\eqref{eq:35} essentially relies on
Proposition~\ref{prop:sol} and is here omitted. We only recall that
the conditions~\ref{thm:ex:i} and~\ref{thm:ex:ii} above are
restated in terms of a primitive $U_o^*$ of $u_o^*$ in~\ref{lem:ex:1}
and~\ref{lem:ex:2} in Lemma~\ref{lem:ExUn}.

\section{Characterizations of $\ii_T (w)$ and $\II_T (W)$}
\label{subs:Charact}

In view of~\cite[Theorem~1.1]{BianchiniLu2012}, for any initial datum
$u_o$ and for all but countably many times $T$, the map $w = S_T u_o$
leads to a function $p$ in $\SBVloc (\reali; \reali)$, see
also~\cite[Theorem~1.2]{AmbrosioDeLellis2004}
and~\cite[Theorem~11.3.5]{DafermosBook}. Therefore, we restrict our
analysis below to functions $w$ such that
$p \in \SBVloc (\reali; \reali)$.

We proceed with our main result, in the version referring to the
Conservation Law~\eqref{eq:1}

\begin{theorem}
  \label{thm:2}
  Let~\eqref{eq:H} hold and $T$ be positive. Fix
  $w \in \L\infty (\reali; \reali)$ such that
  $\ii_T (w) \neq \emptyset$ and $p \in \SBVloc (\reali; \reali)$.
  Then, a map $u_o\in \L\infty (\reali; \reali)$ is in $\ii_T (w)$ if
  and only if the following two conditions hold:
  \begin{enumerate}[label={\sl(\roman*)}]
  \item \label{thm:2:i} for all $x \in \reali$ such that $p$ is
    differentiable at $x$ and $p' (x) \neq 0$,
    \begin{equation}
      \label{eq:11}
      \lim_{y\to x} \frac{1}{p (y) - p (x)} \int_{p (x)}^{p (y)} u_o (\xi) \d\xi
      =
      w (x) \,;
    \end{equation}
  \item \label{thm:2:ii} for all $x \in \reali$ such that
    $w (x-) \neq w (x+)$, for all $v \in [w(x+),w(x-)]$,
    \begin{displaymath}
      \begin{array}{@{}rcl@{}}
        \displaystyle
        \frac{1}{T} \int_{x-T f'(v)}^{x-T f'(w(x+))}{u_o(\xi)\d{\xi}}
        & \leq
        & \left(v \, f'(v)-f(v)\right)
          -
          \left(
          w(x+) \, f'\left(w(x+)\right)-f\left(w(x+)\right)
          \right) \,,
        \\[10pt]
        \displaystyle
        \frac{1}{T} \int_{x-T f'(w(x-))}^{x-T f'(v)}{u_o(\xi)\d{\xi}}
        & \geq
        & \left(w(x-) \, f'\left(w(x-)\right)-f\left(w(x-)\right)\right)
          -
          \left(v \, f'(v)-f(v)\right) \,.
      \end{array}
    \end{displaymath}
  \end{enumerate}
\end{theorem}

Note that~\ref{thm:2:i} holds, in particular, whenever
$p (x)$ is a Lebesgue point of $u_o$. Moreover, at~\ref{thm:2:i}, we
mean both that the limit in the left hand side of~\eqref{eq:11} exists
and that its value is $w (x)$. With reference to
Proposition~\ref{prop:D}, $X_i$ in~\eqref{eq:22} is the set where the
values of $u_o$ are constrained by~\ref{thm:2:i} and,
similarly, $X_{ii}$ is the set where the values of $u_o$ are
constrained by~\ref{thm:2:ii}.

As a side remark note that, as is to be expected, if the flow $f$ is
varied by any additive constant, both conditions~\ref{thm:2:i}
and~\ref{thm:2:ii} remain unchanged.

Towards a restatement of Theorem~\ref{thm:2} in the case of
Hamilton--Jacobi equation we provide the following Theorem, whose
proof is instrumental in the characterization of $\ii_T (w)$. Therein,
we use the Legendre transform $f^*$ of $f$, see
Proposition~\ref{prop:1} for the precise definition.

\begin{theorem}
  \label{thm:3}
  Let~\eqref{eq:H} hold, $T$ be positive an $p$ be as
  in~\eqref{eq:15}. Fix $w \in \L\infty (\reali; \reali)$ such that
  $\ii_T (w) \neq \emptyset$ and $w \in \SBVloc (\reali;
  \reali)$. Then, a map $U \in \W1\infty (\reali; \reali)$ is such
  that $\partial_x U_o \in \ii_T (w)$ if and only if the following two
  conditions hold:
  \begin{enumerate}[label={\sl(\Roman*)}]
  \item \label{thm:3:1} for all $x \in \reali$ such that $p$ is
    differentiable at $x$ and $p' (x) \neq 0$,
    \begin{equation}
      \label{eq:18}
      \lim_{y \to x}
      \dfrac{U_o\left(p (y)\right) - U_o\left(p (x)\right)}{p (y) - p (x)}
      =
      \partial_x W (x)\,;
    \end{equation}
  \item \label{thm:3:2} for all $x \in \reali$ such that
    $\partial_x W (x-) \neq \partial_x W (x+)$, for all
    $y \in \left]p (x-), p (x+) \right[$,
    \begin{eqnarray*}
      \dfrac{U_o \left(p (x+)\right) - U_o(y)}{T}
      \leq
      f^*\left(\dfrac{x-y}{T}\right)
      -
      f^*\left(\dfrac{x-p (x+)}{T}\right) \,,
      \\
      \dfrac{U_o (y) - U_o\left(p (x-)\right)}{T}
      \geq
      f^*\left(\dfrac{x-p (x-)}{T}\right)
      -
      f^*\left(\dfrac{x-y}{T}\right) \,.
    \end{eqnarray*}
  \end{enumerate}
\end{theorem}

\noindent Here we remark that the conditions in Theorem~\ref{thm:2}
and those in Theorem~\ref{thm:3} are equivalent.

\begin{lemma}
  \label{lem:equivalence}
  Under the assumptions and notations of Theorem~\ref{thm:2} and
  Theorem~\ref{thm:3}, if $U_o \in \W1\infty (\reali; \reali)$ and
  $u_o = \partial_x U_o$, then\vspace{-0.5\baselineskip}
  \begin{center}
    \begin{tabular}{ccc}
      $u_o$ satisfies~\ref{thm:2:i} in Theorem~\ref{thm:2}
      & $\Longleftrightarrow$
      & $U_o$ satisfies~\ref{thm:3:1} in Theorem~\ref{thm:3};
      \\
      $u_o$ satisfies~\ref{thm:2:ii} in Theorem~\ref{thm:2}
      & $\Longleftrightarrow$
      & $U_o$ satisfies~\ref{thm:3:2} in Theorem~\ref{thm:3}.
    \end{tabular}
  \end{center}
\end{lemma}

However, the former CL--formulation leads to an easier proof of the
necessity condition, while the latter integral formulation leads to a
simpler verification of the sufficiency part. Therefore, in the proofs
of theorems~\ref{thm:2} and~\ref{thm:3} we follow this scheme:
\begin{displaymath}
  \begin{array}{ccccc}
      & u_o \in \ii_T (w)
      & \iff
      & \partial_x U_o \in \ii_T (w)
    \\
    \mbox{Theorem~\ref{thm:2}}
      & \big\Downarrow
      &
      & \big\Uparrow
      & \mbox{Theorem~\ref{thm:3}}
    \\
      & u_o \mbox{ satisfies~\ref{thm:2:i}}
      & \iff
      & U_o \mbox{ satisfies~\ref{thm:3:1}}
    \\
      & u_o \mbox{ satisfies~\ref{thm:2:ii}}
      & \iff
      & U_o \mbox{ satisfies~\ref{thm:3:2}}
    \\
      &
      & \mbox{Lemma~\ref{lem:equivalence}}
  \end{array}
\end{displaymath}

On the basis of Theorem~\ref{thm:3}, we now deal with the
Hamilton--Jacobi equation~\eqref{eq:35} and state the characterization
of $\II_T (W)$.

\begin{theorem}
  \label{thm:4}
  Let~\eqref{eq:H} hold and $T$ be positive.  Fix
  $W \in \W1\infty (\reali; \reali)$ such that
  $\II_T (W) \neq \emptyset$ and
  $\partial_x W \in \SBVloc (\reali; \reali)$. Let $p$ be as
  in~\eqref{eq:15}, with $w = \partial_x W$.  Then,
  $U_o \in \II_T (W)$ if and only if the following two conditions
  hold:
  \begin{enumerate}[label={$\sl(\Roman*^{{\scriptscriptstyle HJ}})$}]
  \item \label{thm:4:1} for all $x \in \reali$ such that $p$ is
    differentiable at $x$ and $p' (x) \neq 0$,
    \begin{displaymath}
      \lim_{y \to x}
      \dfrac{U_o\left(p (y)\right) - U_o\left(p (x)\right)}{p (y) - p (x)}
      =
      \partial_x W (x)\,;
    \end{displaymath}
  \item \label{thm:4:2} for all $x \in \reali$ such that
    $\partial_x W (x-) \neq \partial_x W (x+)$,
    \begin{displaymath}
      \begin{array}{r@{\,}c@{\,}l@{\qquad}rcl}
        \forall \,y
        &\in
        &\left]p (x-), p (x+) \right[
        & U_o (y) + T \, f^*\left(\frac{x-y}{T}\right)
        & \geq
        & W (x)
          \,,
        \\
        \forall \,y
        &\in
        &\left\{p (x-), p (x+) \right\}
        & U_o (y) + T \, f^*\left(\frac{x-y}{T}\right)
        & =
        & W (x) \,.
      \end{array}
    \end{displaymath}
  \end{enumerate}
\end{theorem}

\section{Geometric Properties of $\ii_T (w)$}
\label{subs:Geo}

On the basis of the characterization provided by Theorem~\ref{thm:2}
and Theorem~\ref{thm:3}, we obtain the following information on
topological and geometrical properties of the set $\ii_T (w)$.

\begin{proposition}
  \label{prop:top}
  Let~\eqref{eq:H} hold and $T$ be positive. Fix
  $w \in \L\infty (\reali; \reali)$ such that
  $\ii_T (w) \neq \emptyset$ and $w \in \SBVloc (\reali; \reali)$.
  Then, with respect to the $\Lloc1$ topology, the set $\ii_T (w)$ is:
  \begin{enumerate}[label={\sl(T\arabic*)}]
  \item \label{prop:top:1} closed;
  \item \label{prop:top:2} with empty interior.
  \end{enumerate}
\end{proposition}

\begin{proposition}
  \label{prop:geo}
  Let~\eqref{eq:H} hold and $T$ be positive. Fix
  $w \in \L\infty (\reali; \reali)$ such that
  $\ii_T (w) \neq \emptyset$ and $w \in \SBVloc (\reali; \reali)$.
  Then,
  \begin{enumerate}[label={\sl(G\arabic*)}]
  \item \label{prop:geo:1} the set $\ii_T (w)$ reduces to a singleton
    if and only if $w \in \C0 (\reali; \reali)$;
  \item \label{prop:geo:2} the set $\ii_T (w)$ is a convex cone having
    as unique extremal point at its vertex the map $u_o^*$ defined in
    Theorem~\ref{thm:extremal}.
  \item \label{prop:geo:3} if $u_o \in \ii_T (w)$ and
    $u_o \neq u_o^*$, then for any $n \in \naturali \setminus\{0\}$
    there exist $v_0, v_1, \ldots, v_N \in \ii_T (w)$ such that
    \begin{equation}
      \label{eq:10}
      u_o = \dfrac{1}{N+1} \, \sum_{i=0}^N v_i
    \end{equation}
    and $v_1-v_0, v_2-v_0, \ldots, v_N-v_0$ are linearly independent.
  \end{enumerate}
\end{proposition}

\noindent Above, by \emph{singleton} we mean up to equality a.e.~or,
equivalently, that the \emph{precise representative} is unique. By
precise representative of $u_o \in \L\infty (\reali; \reali)$ we mean
that
\begin{equation}
  \label{eq:37}
  u_o (x)
  =
  \left\{
    \begin{array}{ll}
      \lim_{r\to 0} \frac{1}{r} \, \int_x^{x+r} u_o (\xi) \, \d\xi
      & \mbox{whenever this limit exists,}
      \\
      0
      & \mbox{otherwise.}
    \end{array}
  \right.
\end{equation}

\section{Proofs Related to~\S~\ref{subs:Notation}}
\label{sec:proof:Notation}

The Lebesgue measure in $\reali$ is denoted by $\mathcal{L}$. Given
$u \in \Lloc\infty (\reali; \reali)$, we define the set $\Leb (u)$ of
its Lebesgue points as the set of those $x \in \reali$ such that
$\lim_{r\to 0} \frac{1}{r} \int_x^{x+r} \modulo{u (\xi) - u (x)} \d\xi
= 0$.  By~\cite[Corollary~2, Chapter~1, \S~7]{EvansGariepy}),
$\mathcal{L} \left(\reali \setminus \Leb (u)\right) = 0$.

Below, we often use the decomposition $u = u_{ac} + u_j + u_c$ of a
$\BV$ function $u$ into its absolutely continuous part $u_{ac}$, its
jump part $u_j$, which is a possibly infinite sum of Heaviside
functions, and its Cantor part $u_c$. Whenever $u_c = 0$, we say that
$u \in \SBV (\reali; \reali)$. Recall that if
$u \in \BV (\reali; \reali)$, then its weak derivative $D u$ is a
Radon measure~\cite[\S~1.1]{EvansGariepy} that admits the
decomposition $D u = (D u)_{ac} + (D u)_j + (D u)_c$, $(D u)_{ac}$
being absolutely continuous with respect to $\mathcal{L}$, $(D u)_j$
is a, possibly infinite, sum of Dirac deltas and $(D u)_c$ is the
Cantor part of $Du$. As is well
known~\cite[Corollary~3.33]{AmbrosioFuscoPallara}, up to sets of
Lebesgue measure $0$,
\begin{displaymath}
  (D u)_{ac} = D(u_{ac}) \,,\quad
  (D u)_j = D(u_j) \;\; \mbox{ and } \;\;
  (D u)_c = D(u_c) \,.
\end{displaymath}
We also denote by $u'$ the density of $(Du)_{ac}$ with respect to the
Lebesgue measure, so that $(D u)_{ac} = u' \, \mathcal{L}$ and
$u' = (u_{ac})'$.  By~\cite[Theorem~3.28]{AmbrosioFuscoPallara} for
a.e.~$x \in \reali$, $u' (x)$ coincides with the limit of the
incremental ratio of $u$ or $u_{ac}$ at $x$.

\smallskip

For later use, we need the following variation of the Area Formula,
see e.g.~\cite[\S~2.10]{AmbrosioFuscoPallara}.

\begin{lemma}
  \label{lem:1}
  Let $\phi \in \SBV(\reali; \reali)$ be weakly increasing. Then, for
  any measurable set $E$,
  \begin{displaymath}
    \int_E \phi_{ac}' (\xi) \d\xi
    =
    \int_\reali
    \card\left(E \cap \phi^{-1} (\xi)\right) \d\xi
  \end{displaymath}
  with $\phi'_{ac}$ being the absolutely continuous part of $\phi'$.
\end{lemma}

\begin{proof}
  Throughout this proof, $A^c$ is the complement of the set $A$ in
  $\reali$.  Denote by $H_\xi$ the Heaviside function centered at
  $\xi$, i.e., $H_\xi (x) = 1$ for $x \geq \xi$ and $H_\xi (x) = 0$
  for $x < \xi$.

  Define $\phi_s = \sum_n \alpha_n \, H_{\xi_n}$, with
  $\alpha_n = \phi (\xi_n+) - \phi (\xi_n-)$ and
  $\{\xi_n \colon n \in \naturali\}$ being the set of points of jump
  in $\phi$.  Since $\phi \in \SBV (\reali; \reali)$,
  by~\cite[Chapter~2, Section~25]{RieszNagy}, the function
  $\phi_{ac} = \phi -\phi_s$ is in $\AC (\reali; \reali)$.

  For any measurable set $E$, define
  $\mu (E) = \int_E \phi_{ac}' (\xi) \d\xi$. By construction, $\mu$ is
  a measure and for any $a,b \in \reali$ with $a \leq b$, we have that
  $\mu ([a,b]) = \phi_{ac} (b) - \phi_{ac} (a)$. Hence, choosing $a$
  and $b$ among the continuity points of $\phi$, we have that
  \begin{equation}
    \label{eq:16}
    \mu ([a,b])
    =
    \phi (b) - \phi (a) - \sum_{n \colon \xi_n \in \left]a,b\right[} \alpha_n \,.
  \end{equation}
  Define, for any measurable set $E$, also
  $\nu (E) := \int_{\reali} \card \left(E \cap \phi^{-1} (\xi)\right)
  \d\xi$.  The set function $\nu$ is a measure, as it follows from the
  Monotone Convergence Theorem, see
  e.g.~\cite[Theorem~1.19]{AmbrosioFuscoPallara}, and from the
  countable additivity of the counting measure.

  Denote
  $A := \bigcup_{n \in \naturali} \left]\phi (\xi_n-), \phi
    (\xi_n+)\right[$.  For any $a,b \in \reali$, with $a \leq b$ being
  continuity points of $\phi$, by the monotonicity of $\phi$ note that
  \begin{displaymath}
    \card\left( [a,b] \cap \phi^{-1} (\xi)\right)
    =
    \left\{
      \begin{array}{l@{\qquad}l}
        0
        & \xi < \phi(a) \mbox{ or } \xi > \phi(b) \mbox{ or } \xi \in A\,,
        \\
        1
        & \phi^{-1}(\xi) \mbox{ is a singleton in }[a,b] \cap A^c\,,
        \\
        +\infty
        & \mbox{otherwise.}
      \end{array}
    \right.
  \end{displaymath}
  Indeed, if $ [a,b] \cap \phi^{-1} (\xi)$ contains two points, say
  $\xi_1$ and $\xi_2$ with $\xi_1 < \xi_2$, then
  $[\xi_1, \xi_2] \subseteq \left( [a,b] \cap \phi^{-1} (\xi)\right)$,
  so that $\card\left( [a,b] \cap \phi^{-1} (\xi)\right) = +\infty$
  and there exist only countably many such points. As a consequence,
  \begin{displaymath}
    \nu ([a,b])
    =
    \int_\reali \card\left( [a,b] \cap \phi^{-1} (\xi)\right) \d\xi
    =
    \int_{\reali} \caratt{[\phi(a), \phi(b)] \cap A^c} (\xi ) \, \d\xi
    =
    \mathcal{L} ([\phi(a), \phi (b)] \cap A^c)
  \end{displaymath}
  so that
  \begin{eqnarray*}
    [\phi(a), \phi (b)] \cap A^c
    & =
    & [\phi(a), \phi (b)] \cap
      \left(\bigcup_{n \in \naturali} \left] \phi (\xi_n-), \phi (\xi_n+) \right[\right)^c
    \\
    & =
    &  [\phi(a), \phi (b)] \setminus
      \bigcup_{n \colon \xi_n \in \left]a, b\right[}
      \left] \phi (\xi_n-), \phi (\xi_n+) \right[
  \end{eqnarray*}
  and the latter union in the right hand side above is contained in
  $[\phi(a), \phi (b)]$, due to our choice of $a$ and $b$. Passing to
  the Lebesgue measure of the sets on the two sides of the latter
  equality,
  \begin{eqnarray}
    \nonumber
    \nu \left([a,b]\right)
    & =
    & \mathcal{L} \left([\phi(a), \phi (b)] \cap A^c\right)
    \\
    \nonumber
    & =
    & \mathcal{L}\left([\phi(a), \phi (b)] \mathbin{\big\backslash}
      \bigcup_{n \colon \xi_n \in \left]a, b\right[}
      \left] \phi (\xi_n-), \phi (\xi_n+) \right[\right)
    \\
    \nonumber
    & =
    & \mathcal{L} \left([\phi(a), \phi (b)]\right)
      -
      \sum_{n \colon \xi_n \in \left]a, b\right[}
      \mathcal{L} \left(\left] \phi (\xi_n-), \phi (\xi_n+) \right[\right)
    \\
    \label{eq:19}
    & =
    & \phi (b) - \phi(a)
      -
      \sum_{n \colon \xi_n \in \left]a, b\right[} \alpha_n
  \end{eqnarray}
  By~\eqref{eq:16} and~\eqref{eq:19} we have that $\mu = \nu$ on all
  intervals $[a,b]$ with $a$, $b$ continuity points of $\phi$. The
  choice of $\phi$ ensures that these points are dense in $\reali$,
  completing the proof.
\end{proof}

\begin{proofof}{Proposition~\ref{prop:D}}
  Introduce the sets
  \begin{displaymath}
    P_1
    :=
    \left\{
      x \in \reali \colon p\mbox{ or } p_{ac}
      \mbox{ is not differentiable at }x
    \right\}
    \quad \mbox{ and } \quad
    P_2
    :=
    \left\{
      x \in \reali \setminus P_1 \colon p' (x) = 0
    \right\} \,.
  \end{displaymath}
  Remark that if $x \in P_2$, then $p$, $p_{ac}$ and $p_j$ are all
  differentiable at $x$ and $p'_{ac} (x) = p'_j (x) = 0$, since $p$,
  $p_{ac}$ and $p_j$ are all weakly increasing by~\ref{thm:1:2} in
  Theorem~\ref{thm:1}.

  By~\cite[Chapter~1, Section~2]{RieszNagy}, $P_1$ has Lebesgue
  measure $0$ and $\int_{P_1 \cup P_2} p_{ac}' (x) \d{x} =0$. Note
  that Lemma~\ref{lem:1}, which can be applied since
  $p \in \SBVloc (\reali; \reali)$, ensures that
  $\int_{P_1 \cup P_2} p_{ac}' (x) \d{x} = \int_{\reali}
  \card\left((P_1 \cup P_2) \cap p^{-1} (\xi)\right) \d\xi$, so that
  $\card\left((P_1 \cup P_2) \cap p^{-1} (\xi)\right) =0$ for
  a.e.~$\xi \in \reali$ and, equivalently, $p (P_1 \cup P_2)$ is
  negligible, i.e.,
  \begin{equation}
    \label{eq:21}
    \mathcal{L}\left(p (P_1 \cup P_2)\right) = 0 \,.
  \end{equation}
  Observe that
  \begin{equation}
    \label{eq:25}
    p (\reali) = X_i \cup p (P_1 \cup P_2) \,.
  \end{equation}
  By~\eqref{eq:O}, the function $p$ is non decreasing, hence $X_{ii}$
  is an at most countable union of non empty, disjoint and open
  intervals.  By the properties of backward characteristics, the set
  $\reali \setminus \left(X_{ii} \cup p (\reali) \right)$ is at most
  countable, so that
  \begin{equation}
    \label{eq:20}
    \mathcal{L}
    \left(\reali \setminus \left(X_{ii} \cup p (\reali) \right)\right) = 0 \,.
  \end{equation}
  Using the relations above,
  \begin{displaymath}
    \begin{array}{@{}rcl@{\qquad}l@{}}
      \reali \setminus (X_i \cup X_{ii})
      & \subseteq
      & \reali \setminus \left(X_i \cup X_{ii} \cup p (P_1 \cup P_2)\right)
        \cup p (P_1 \cup P_2)
      & \mbox{[by~\eqref{eq:25}]}
      \\
      & \subseteq
      & \reali \setminus \left(X_{ii} \cup p (\reali)\right)
        \cup p (P_1 \cup P_2)
      \\
      \mathcal{L}\left(\reali \setminus (X_i \cup X_{ii})\right)
      & \leq
      & \mathcal{L}\left(\reali \setminus \left(X_{ii} \cup p (\reali)\right)\right)
        + \mathcal{L} \left(p (P_1 \cup P_2)\right)
      & \mbox{[by~\eqref{eq:21} and~\eqref{eq:20}]}
      \\
      & \leq
      & 0 \,,
    \end{array}
  \end{displaymath}
  completing the proof.
\end{proofof}

The following classical result is of use in the subsequent proof as
well as in what follows,

\begin{proposition}[{\cite[Lemma~3.2]{Dafermos77}}]
  \label{prop:4}
  Let $f$ satisfy~\eqref{eq:H} and let $u$ be a weak entropy solution
  to~\eqref{eq:1}. Let $a, b \in [0,T]$ with $a < b$ and choose two
  maps $\alpha, \beta \in \C{0,1} ([a,b]; \reali)$ with
  $\alpha \leq \beta$. Then, for a.e.~$t_1, t_2 \in [a,b]$ with
  $t_1 \leq t_2$,
  \begin{eqnarray*}
    \int_{\alpha (t_2)}^{\beta (t_2)} u (t_2,x) \d{x}
    -
    \int_{\alpha (t_1)}^{\beta (t_1)} u (t_1,x) \d{x}
    & =
    & \int_{t_1}^{t_2} \left(
      f\left(u (t, \alpha (t)-)\right) - \dot \alpha (t) \, u (t, \alpha (t)-)
      \right)
      \d{t}
    \\
    &
    & -
      \int_{t_1}^{t_2} \left(
      f\left(u (t, \beta (t)+)\right) - \dot \beta (t) \, u (t, \beta (t)+)
      \right)
      \d{t} \,.
  \end{eqnarray*}
\end{proposition}

Remark that if $\gamma$ is a Lipschitz curve, then
\begin{displaymath}
  \int_{t_1}^{t_2} \left(
    f\left(u (t, \gamma (t)-)\right) - \dot \gamma (t) \, u (t, \gamma (t)-)
  \right)
  \d{t}
  =
  \int_{t_1}^{t_2} \left(
    f\left(u (t, \gamma (t)+)\right) - \dot \gamma (t) \, u (t, \gamma (t)+)
  \right)
  \d{t} \,,
\end{displaymath}
as it follows from Proposition~\ref{prop:4} in the case
$\alpha = \beta = \gamma$.that

\begin{proofof}{Proposition~\ref{prop:sol}}
  Introduce the map $\tilde U (t,x) = (\SS_t U_o) (x)$. We prove that
  $\tilde U$ coincides with the map defined in~\eqref{eq:31}.
  By~\cite[Theorem~1.1]{KarlsenRisebro2002}, the map
  $\partial_x \tilde U$ solves the Conservation Law~\eqref{eq:1}, so
  that
  \begin{displaymath}
    \partial_x \tilde U (t,x)
    =
    u (t,x)
    =
    \partial_x \int_{\gamma (t)}^x u (t,\xi) \d\xi \,.
  \end{displaymath}
  Hence, there exist a map $\Upsilon \in \W1\infty ([0,T]; \reali)$
  with $\Upsilon (0) = 0$ and a $c \in \reali$ such that
  \begin{displaymath}
    \tilde U (t,x)
    =
    \int_{\gamma (t)}^x u (t,\xi) \d\xi + \Upsilon (t) + c
  \end{displaymath}
  and, using Proposition~\ref{prop:4} with $\alpha (t) = \gamma (t)$,
  $\beta (t) = x$, $t_2=t$ and $t_1=0$,
  \begin{eqnarray*}
    &
    & \tilde U (t, x) - \tilde U_o (x)
    \\
    & =
    & \int_{\gamma (t)}^x u (t,\xi) \d\xi
      -
      \int_{0}^x u (t,\xi) \d\xi
      + \Upsilon (t)
    \\
    & =
    & \int_{0}^{t}
      \left(
      f\left(u\left(\tau,\gamma (\tau)-\right)\right)
      -
      \dot\gamma (\tau) \, u\left(\tau, \gamma (\tau)-\right)
      \right) \d\tau
      -
      \int_{0}^{t}
      f\left(u(\tau,x+)\right)
      \d\tau
      + \Upsilon (t)
  \end{eqnarray*}
  By the differentiability properties of $\tilde U$,
  see~\cite[Theorem~1, Section~10.1.2]{Evans}, we have
  \begin{equation}
    \label{eq:36}
    \tilde U (t, x) - \tilde U_o (x)
    =
    -\int_{0}^{t} f\left(\partial_x \tilde U (\tau, x)\right) \d\tau
    =
    -\int_{0}^{t} f\left(u (\tau, x)\right) \d\tau \,.
  \end{equation}
  So that
  \begin{displaymath}
    \Upsilon (t)
    =
    \int_{0}^{t}
    \left(
      \dot\gamma (\tau) \, u\left(\tau, \gamma (\tau)-\right)
      -
      f\left(u\left(\tau,\gamma (\tau)-\right)\right)
    \right) \d\tau
  \end{displaymath}
  completing the proof.
\end{proofof}

\section{Proofs Related to \S~\ref{subs:NonEmpty}}
\label{sec:proof:NonEmpty}

A tool used below is the following classical representation
formula for the solutions to~\eqref{eq:1}.

\begin{proposition}[{\cite[Theorem~2.1]{Lax}}]
  \label{prop:1}
  Let~\eqref{eq:H} hold and $u_o \in \L1 (\reali; \reali)$. The
  solution to
  \begin{equation}
    \label{eq:6}
    \left\{
      \begin{array}{l}
        \partial_t u + \partial_x f (u) = 0
        \\
        u (0,x) = u_o (x)
      \end{array}
    \right.
  \end{equation}
  in the sense of Definition~\ref{def:HCL} is the map
  \begin{equation}
    \label{eq:7}
    u (t,x) = g\left(\frac{x-y (t,x)}{t}\right)
    \quad \mbox{ where } \quad
    \begin{array}{l}
      y (t,x) \mbox{ minimizes } y \to s (t,x,y)\,,
      \\
      s (t,x,y) = t\,
      f^*\!\left(\frac{x-y}{t}\right) + \int_0^y u_o (\xi) \d\xi \,,
      \\
      f^* (\lambda) = \lambda \, g (\lambda) - f\left(g (\lambda)\right) \,,
      \\
      g (\lambda) = (f')^{-1} (\lambda) \,.
    \end{array}
  \end{equation}
\end{proposition}

Formula~\eqref{eq:7} is a direct consequence of the classical
Lax--Hopf formula, see~\cite[Theorem~2.1]{Lax}, \cite{Hopf1950,
  Lax1954} or~\cite[\S~11.4]{DafermosBook}, adapted to the present
assumption~\eqref{eq:H} on $f$. Here we only remark that $y$ is
uniquely defined for a.e.~$x \in \reali$.

We pass to some properties on the solution to the conservation law
obtained reversing space in~\eqref{eq:1} and assigning $w$ as initial
datum.

\begin{lemma}
  \label{lem:u0}
  Let~\eqref{eq:H} hold and $T$ be positive. Fix
  $w \in \L\infty (\reali; \reali)$ such that~\eqref{eq:O} holds. Call
  $\tilde u$ the weak entropy solution to
  \begin{equation}
    \label{eq:5}
    \left\{
      \begin{array}{l}
        \partial_\tau {\tilde u} + \partial_\xi f ({\tilde u}) = 0
        \\
        {\tilde u} (0,\xi) = w (-\xi)
      \end{array}
    \right.
  \end{equation}
  and define $u_o^* (x) := \tilde u (T,-x)$. Then:
  \begin{enumerate}[label={\sl(R\arabic*)}]
  \item \label{u0:it:1} The map $\tilde u$ is Lipschitz continuous on
    any compact subset of $\left]0, T\right[ \times \reali$.
  \item \label{u0:it:2} The map $u_o^*$ is one sided Lipschitz and in
    $\BV (\reali; \reali)$;
  \item \label{u0:it:3} The map $u_o^*$ satisfies
    $u_o^* \in \ii_T (w)$.
  \end{enumerate}

\end{lemma}

\begin{proof}
  By~\eqref{eq:H} and~\cite[Theorem~11.2.1]{DafermosBook}, for any
  $\tau \in \left]0, T\right]$, $\xi \in \reali$ and $h>0$, we have
  \begin{equation}
    \label{eq:12}
    f'\left({\tilde u} (\tau, \xi+h)\right) - f'\left({\tilde u} (\tau, \xi)\right)
    \leq
    \dfrac{h}{\tau} \,.
  \end{equation}
  On the other hand, for any $\tau \in \left[0, T\right[$,
  $\xi \in \reali$ and $h>0$, introduce the values attained at
  $\tau=0$ by the minimal backward characteristics originating from
  $(\tau, \xi+h)$ and $(\tau, \xi)$:
  \begin{equation}
    \label{eq:13}
    \xi_r = \xi+h-\tau \, f'\left({\tilde u} (\tau, \xi+h)\right)
    \quad \mbox{ and } \quad
    \xi_\ell = \xi-\tau \, f'\left({\tilde u} (\tau, \xi)\right) \,,
  \end{equation}
  so that
  \begin{displaymath}
    \begin{array}{@{}rcl@{\qquad\quad}r@{}}
      &
      & f'\left({\tilde u} (\tau, \xi+h)\right) - f'\left({\tilde u} (\tau, \xi)\right)
      \\
      & \geq
      & f'\left({\tilde u}_o \left(\xi_r-\right)\right)
        -
        f'\left({\tilde u}_o \left(\xi_\ell+\right)\right)
      &\mbox{\cite[Theorem~11.1.3]{DafermosBook}}
      \\
      & \geq
      & - \dfrac{\xi_r - \xi_\ell}{T}
      & \mbox{[by Oleinik condition~\eqref{eq:O}]}
      \\
      & =
      & - \dfrac{h - \tau \left(
        f'\left({\tilde u} (\tau, \xi+h)\right) - f'\left({\tilde u} (\tau, \xi)\right)
        \right)}{T}
      & \mbox{[by~\eqref{eq:13}]}
    \end{array}
  \end{displaymath}
  hence
  \begin{equation}
    \label{eq:14}
    f'\left({\tilde u} (\tau, \xi+h)\right) - f'\left({\tilde u} (\tau, \xi)\right)
    \geq
    - \dfrac{h}{T-\tau} \,.
  \end{equation}
  The two inequalities~\eqref{eq:12} and~\eqref{eq:14} ensure that
  $x \to f'\left({\tilde u} (\tau, x)\right)$ is Lipschitz continuous
  for $\tau \in \left]0, T\right[$, a Lipschitz constant being
  $\max\left\{1/\tau \, ,\; 1/ (T-\tau)\right\}$. Therefore,
  by~\eqref{eq:5} and~\eqref{eq:H}, $\partial_\tau {\tilde u}$ is in
  $\L\infty$, proving~\ref{u0:it:1} and~\ref{u0:it:2}

  For any $\C1$ entropy -- entropy flux pair $(\eta, q)$
  for~\eqref{eq:5}, see~\cite[Chapter~3, \S~2]{DafermosBook}, we have
  \begin{displaymath}
    \partial_\tau \eta ({\tilde u}) + \partial_\xi q ({\tilde u}) = 0
    \qquad \mbox{a.e.~in } [0,T] \times \reali \,.
  \end{displaymath}
  Passing from the $(\tau, \xi)$ to the $(t,x)$ variables and setting
  \begin{equation}
    \label{eq:4}
    \begin{array}{rcl}
      t
      & :=
      & T-\tau
      \\
      x
      & :=
      & -\xi
    \end{array}
    \qquad\qquad u (t,x) := {\tilde u} (T-t, -x) \,,
  \end{equation}
  we obtain that $\partial_t \eta(u) + \partial_x q (u)=0$ in
  distributional sense. By~\cite[Chapter~6, \S~2]{DafermosBook}, $u$
  is a weak entropy solution to~\eqref{eq:1} such that $u (T) = w$ and
  hence $u_o^* \in \ii_T (w)$, proving~\ref{u0:it:3}.
\end{proof}

\begin{lemma}
  \label{lem:ExUn}
  Let~\eqref{eq:H} hold and $T$ be positive. Fix
  $w \in \L\infty (\reali; \reali)$ such that
  $\ii_T (w) \neq \emptyset$ and $p \in \SBVloc (\reali; \reali)$.
  Then, there exists a unique
  $u_o^\flat \in \L\infty (\reali; \reali)$ such that any of its
  primitives $U_o^\flat$ satisfies
  \begin{enumerate}[label={\sl(\Roman*\/$^\flat$)}]
  \item \label{lem:ex:1} for all $x \in \reali$ such that $p$ is
    differentiable at $x$ and $p' (x) \neq 0$,
    \begin{displaymath}
      \lim_{y \to x}
      \dfrac{U_o^\flat\left(p (y)\right) - U_o^\flat\left(p (x)\right)}%
      {p (y) - p (x)}
      =
      w (x)\,;
    \end{displaymath}
  \item \label{lem:ex:2} for all $x \in \reali$ such that
    $w (x-) \neq w (x+)$, for all
    $y \in \left]p (x-), p (x+) \right[$,
    \begin{eqnarray*}
      \dfrac{U_o^\flat \left(p (x+)\right) - U_o^\flat(y)}{T}
      =
      f^*\left(\dfrac{x-y}{T}\right)
      -
      f^*\left(\dfrac{x-p (x+)}{T}\right) \,,
      \\
      \dfrac{U_o^\flat (y) - U_o^\flat\left(p (x-)\right)}{T}
      =
      f^*\left(\dfrac{x-p (x-)}{T}\right)
      -
      f^*\left(\dfrac{x-y}{T}\right) \,.
    \end{eqnarray*}
  \end{enumerate}
\end{lemma}

\begin{proof}
  We prove separately existence and uniqueness, referring to
  $X_i$ and $X_{ii}$ defined in~\eqref{eq:22}.

  \paragraph{Existence.} Let $u_o \in \ii_T (w)$ and call $U_o$ any of
  its primitives. For all $y \in \reali \setminus X_{ii}$, define
  $U_o^\flat (y) = U_o (y)$ so that~\ref{lem:ex:1} holds since
  $p (\reali) \subseteq \reali \setminus X_{ii}$.

  Consider now a $y \in X_{ii}$. Then, there exists a unique
  $x \in \reali$ such that $y \in \left]p (x-), p
    (x+)\right[$. Define, for all $y \in [p (x-), p (x+)]$,
  \begin{equation}
    \label{eq:33}
    U_o^\flat (y)
    :=
    U_o\left(p (x-)\right)
    + T \left(
      f^*\left(\dfrac{x-p (x-)}{T}\right)
      - f^*\left(\dfrac{x-y}{T}\right)
    \right) \,.
  \end{equation}
  Straightforward computations show that
  $U_o^\flat (p (x-)) = U_o\left(p (x-)\right)$ and that~\ref{lem:ex:2} holds.

  \paragraph{Uniqueness.}  Fix $\bar y \in X_i \cap \Leb (u_o^\flat)$.
  There exists an $\bar x \in \reali$ such that $p (\bar x) = \bar y$,
  $p$ is differentiable at $\bar x$ and $p' (\bar x) > 0$, then
  \begin{displaymath}
    \lim_{y \to x}
    \dfrac{U_o^\flat\left(p (y)\right) - U_o^\flat\left(p (\bar x)\right)}%
    {p (y) - p (\bar x)}
    =
    \left\{
      \begin{array}{l@{\qquad}l}
        w (\bar x)
        & \mbox{[by~\ref{lem:ex:1}]}
        \\
        u_o^\flat (\bar y)
        &\mbox{[by the choice of } \bar y \mbox{]}
      \end{array}
    \right.
  \end{displaymath}
  On the other hand, if $\bar y \in X_{ii}$, there exists an
  $\bar x \in \reali$ such that
  $\bar y \in \left]p (\bar x-), p (\bar x+)\right[$ and
  by~\ref{lem:ex:2}, for all
  $y \in \left]p (\bar x-), p (\bar x+)\right[$,
  \begin{displaymath}
    U_o^\flat (y) = U_o^\flat\left(p (\bar x-)\right)
    +
    T \left(
      f^*\left(\dfrac{\bar x-p (\bar x-)}{T}\right)
      -
      f^*\left(\dfrac{\bar x-y}{T}\right)
    \right)
  \end{displaymath}
  so that $U_o^\flat$ is of class $\C1$ on the interval
  $\left]p (\bar x-), p (\bar x+)\right[$, which implies that
  $u_o^\flat (\bar y) = g\left(\frac{\bar x - \bar y}{T}\right)$, with
  $g$ defined as in Proposition~\ref{prop:1}.

  Thus, Proposition~\ref{prop:D} ensures that $u_o^\flat$ is
  a.e.~uniquely defined.
\end{proof}

\begin{lemma}
  \label{lem:equality}
  Let~\eqref{eq:H} hold and $T$ be positive. Fix
  $w \in \L\infty (\reali; \reali)$ such that
  $p \in \SBVloc (\reali; \reali)$ and ~\eqref{eq:O} holds.  Then, the
  function $u_o^*$ defined in Lemma~\ref{lem:u0} and the function
  $u_o^\flat$ defined in Lemma~\ref{lem:ExUn} coincide.
\end{lemma}

\begin{proof}
  It is sufficient to prove that the map $u_o^*$ defined in
  Lemma~\ref{lem:u0} satisfies~\ref{lem:ex:1} and~\ref{lem:ex:2} in
  Lemma~\ref{lem:ExUn}. We refer below to $\tilde{u}$ and $u$ as
  defined in Lemma~\ref{lem:u0} and related as in~\eqref{eq:4}.

  The proof consists of the following steps.

  \paragraph{Generalized Characteristics:} The fact that $\tilde{u}$
  and $u$ are locally Lipschitz on $\left]0,T\right[ \times \reali$
  implies that the generalized characteristics are actually classical
  characteristics and only one of them goes through a point
  $(\tau,\xi)$ or $(t,x)$ when $\tau,t\in \left]0,T\right[$.

  Of course at times $0$ and $T$ it is still possible to have multiple
  characteristics since the semi-Lipschitz property only guarantees
  uniqueness in one direction. Indeed, $u$ or $\tilde u$ may display
  rarefaction waves generated at time $0$ or shock waves forming at
  time $T$.

  Furthermore simple calculations show the equivalence
  through~\eqref{eq:4} of
  \begin{equation}
    \label{eq:V:8}
    \frac{\mathrm{d}\xi}{\mathrm{d}\tau}
    =
    f'(\tilde{u}(\tau,\xi(\tau)))
    \quad \mbox{ and } \quad
    \frac{\mathrm{d}x}{\mathrm{d}t}
    =
    f'(u(t,x(t))).
  \end{equation}

  \paragraph{Properties of $u_0^*$:} Now let us consider a point
  $x\in \reali$ such that $w(x^-)> w(x^+)$. Therefore, the minimal and
  maximal backward characteristics for $u$ through $(T,x)$, say
  $\gamma^-$ and $\gamma^+$, are
  \begin{equation}
    \label{eq:V:10}
    \gamma^+(t) = x - (T-t) \, f'\left(w(x^+)\right)
    \quad \mbox{ and } \quad
    \gamma^-(t) = x - (T-t) \, f'\left(w(x^-)\right) \,.
  \end{equation}
  By~\ref{thm:1:2} in Theorem~\ref{thm:1}, the function $p$ is non
  decreasing, hence
  \begin{equation}
    \label{eq:V:11}
    \forall z\in \reali,\qquad
    p(z) = z - T \, f'\left(w(z)\right) \,,
  \end{equation}
  we have
  \begin{equation}
    \label{eq:V:12}
    \gamma^-(0)=p(x^-)
    \quad \mbox{ and } \quad
    \gamma^+(0)=p(x^+) \,.
  \end{equation}
  Using \eqref{eq:4}, $\gamma^-$ and $\gamma^+$ provide generalized
  characteristics for $\tilde{u}$ through the formul\ae
  \begin{equation}
    \label{eq:V:13}
    \phi^-(\tau)=-\gamma^-(T-\tau)
    \quad \mbox{ and } \quad
    \phi^+(\tau)=-\gamma^+(T-\tau)\,.
  \end{equation}
  Now consider $y\in \left]p(x^-),p(x^+)\right[$. We have of course
  \begin{equation}
    \label{eq:V:14}
    \phi^-(T)>-y>\phi^+(T) \,.
  \end{equation}
  Since $\phi^-(0)=\phi^+(0)=-x$ and the characteristics do not cross
  for $\tau \in \left]0,T\right[$, the minimal and maximal backward
  characteristics through $(T,-y)$ for $\tilde{u}$ are straight lines
  and, in fact, coincide and go through $(0,-x)$. Thus, we have
  \begin{equation}
    \label{eq:V:15}
    \tilde{u} (T,(-y)^+)=\tilde{u} (T,(-y)^-)
    \quad \mbox{ and } \quad
    -y-f'(\tilde{u} (T,-y)) = -x \,.
  \end{equation}
  From which we get that $u_0^*$ is given between $p(x^-)$ and
  $p(x^+)$ by the formula
  \begin{equation}
    \label{eq:V:16}
    \forall y\in \left]p(x^-),p(x^+)\right[,\qquad
    u_0^*(y)=g\left(\frac{x-y}{T}\right),
  \end{equation}
  where $g$ is the reciprocal function of $f'$, as in~\eqref{eq:7}.

  \paragraph{Condition~\ref{lem:ex:1} holds:}
  Fix $x,y \in \reali$ with $x < y$.  Apply Proposition~\ref{prop:4}
  with $a = 0$, $b = T$ and as $\alpha$, respectively $\beta$, the
  minimal, respectively maximal, backward characteristic from $(T,x)$,
  respectively $(T,y)$. Note that our choice
  $u \in \C0 ([0,T]; \L1 (\reali; \reali))$ in
  Definition~\ref{def:HCL} allows to select $t_1 = 0$ and $t_2 = T$.
  Then, with the notation~\eqref{eq:15},
  \begin{eqnarray*}
    &
    & U_o^\flat\left(p (y)\right) - U_o^\flat\left(p (x)\right)
    \\
    & =
    & \int_{p (x)}^{p (y)} u_o (\xi) \d\xi
    \\
    & =
    & \int_x^y w (\xi) \d\xi
      -
      \int_0^t \left(
      f\left(u (t, \alpha (t)-)\right)
      -
      \dot\alpha (t) \, u \left(t, \alpha (t)\right)
      \right) \d{t}
    \\
    &
    & +
      \int_0^t \left(
      f\left(u (t, \beta (t)-)\right)
      -
      \dot\beta (t) \, u \left(t, \beta (t)\right)
      \right) \d{t}
    \\
    & =
    & \int_x^y w (\xi) \d\xi
      \qquad\qquad\qquad\qquad\qquad\qquad\qquad
      \mbox{[by~\cite[Theorem~11.1.3]{DafermosBook}]}
    \\
    &
    & - T\left(f\left(w (x)\right) - w (x) \, f'\left(w (x)\right)\right)
      + T\left(f\left(w (y)\right) - w (y) \, f'\left(w (y)\right)\right) \,,
  \end{eqnarray*}
  and entirely similar equalities hold if $y<x$. Let now $x \in E$, so
  that $p$, and hence $w$, are differentiable at $x$.  In the previous
  equality, divide by $y-x$ and pass to the limit as $y \to x$. Then,
  the right hand side converges and, hence, also the left hand
  side. We thus obtain
  \begin{equation}
    \label{eq:3:moved}
    \lim_{y \to x}
    \dfrac{U_o^\flat\left(p (y)\right) - U_o^\flat\left(p (x)\right)}{y-x}
    =
    w (x)
    - T \, w (x) \; w' (x) \; f''\left(w (x)\right)
    =
    w (x) \; p' (x) \,.
  \end{equation}
  Since $p$ is differentiable at $x$, $p' (x) \neq 0$ and
  using~\eqref{eq:3:moved}, we prove the existence of the following
  limit, at the same time computing its value:
  \begin{eqnarray*}
    \lim_{y \to x}
    \dfrac{U_o^\flat\left(p (y)\right) - U_o^\flat\left(p (x)\right)}{p(y)-p(
    x)}
    & =
    &
      \lim_{y\to x}
      \dfrac{1}{p (y)- p (x)}
      \int_{p (x)}^{p (y)} u_o (\xi) \d\xi
    \\
    & =
    &  \lim_{y\to x}
      \dfrac{\displaystyle \dfrac{1}{y-x} \int_{p (x)}^{p (y)} u_o (\xi) \d\xi}{\displaystyle \dfrac{p (y)- p (x)}{y-x}}
    \\
    & =
    & \dfrac{\displaystyle \lim_{y\to x} \dfrac{1}{y-x} \int_{p (x)}^{p (y)} u_o (\xi) \d\xi}{\displaystyle \lim_{y\to x}\dfrac{p (y)- p (x)}{y-x}}
    \\
    & =
    & \dfrac{w (x) \; p' (x)}{p' (x) }
    \\
    & =
    & w (x)
  \end{eqnarray*}
  which completes the proof of~\ref{lem:ex:1}.

  \paragraph{Condition~\ref{lem:ex:2} holds:}Still using the same
  point $x$, we consider the function $\Upsilon$ given by
  \begin{equation}
    \label{eq:V:17}
    \forall v\in \left]w(x^+),w(x^-)\right[,\qquad
    \Upsilon(v)
    :=
    \frac{1}{T} \int_{x-T f'(w(x^-))}^{x-T f'(v)} u_0^*(y) \d{y} \,.
  \end{equation}
  Thanks to~\eqref{eq:V:16}, $u_0^*$ is continuous on
  $\left]p(x^-),p(x^+)\right[$.  Therefore, $\Upsilon$ is actually of
  class $\C1$ on $\left]w(x^+),w(x^-)\right[$ and we have, using
  again~\eqref{eq:V:16},
  \begin{equation}
    \label{eq:V:18}
    \forall v\in \left]w(x^+),w(x^-)\right[,\qquad
    \Upsilon'(v) = -v\,f''(v).
  \end{equation}
  Since obviously $\Upsilon(w(x^-))=0$ we see that for any $v$ in
  $\left]w(x^+),w(x^-)\right[$,
  \begin{equation}
    \label{eq:V:19}
    \frac{1}{T} \int_{x-T f'(w(x^-))}^{x-T f'(v)} \! u_0^*(y) \d{y}
    =
    \Upsilon(v)
    =
    \left(f(v)-v\,f'(v)\right)-\left(f(w(x^-))-w(x^-)f'(w(x^-))\right) \!.
  \end{equation}
  In the same way, we can also show that for any $v$ in
  $\left]w(x^+),w(x^-)\right[$,
  \begin{equation}
    \label{eq:V:20}
    \frac{1}{T} \int_{x-T\,f'(v)}^{x-T\,f'(w(x^+))} u_0^*(y) \d{y}
    =
    \left(f(w(x^+))-w(x^+)f'(w(x^+))\right)-\left(f(v)-v\,f'(v)\right) \,,
  \end{equation}
  completing the proof.
\end{proof}

\begin{proofof}{Theorem~\ref{thm:extremal}}
  Lemma~\ref{lem:u0} ensures that the map $u_o^*$ defined
  in~\ref{thm:ex:1} is one sided Lipschitz and satisfies
  $u_o^* \in \ii_T (w)$. Lemma~\ref{lem:ExUn} ensures that there
  exists a unique $u_o^*$
  satisfying~\ref{thm:ex:3}. Lemma~\ref{lem:equality} shows that these
  two functions coincide, completing the proof.
\end{proofof}

\begin{proofof}{Corollary~\ref{thm:1}}
  The implication~\ref{thm:1:1}~$\Rightarrow$~\ref{thm:1:2} holds
  by~\cite[Theorem~11.2.1]{DafermosBook}. The
  converse~\ref{thm:1:2}~$\Rightarrow$~\ref{thm:1:1} follows
  from~\ref{thm:exi:1} in Theorem~\ref{thm:extremal},
  or~Lemma~\ref{lem:u0}.
\end{proofof}

\section{Proofs Related to \S~\ref{subs:Charact}}
\label{sec:proof:Charact}

\begin{proofof}{Theorem~\ref{thm:2}, (necessity)}
  Let $u_o$ be the precise representative~\eqref{eq:37} of the initial
  datum to~\eqref{eq:1} such that the corresponding solution $u$
  satisfies $u (T) = w$. We now prove that conditions~\ref{thm:2:i}
  and~\ref{thm:2:ii} hold.

  Since $p \in \SBV (\reali; \reali)$, we write below
  $p = p_{ac} + p_s$, with $p_{ac} \in \AC (\reali; \reali)$ and $p_s$
  being a sum of countably many Heaviside functions centered at the
  points of jump in $w$. Note that $p$, $p_{ac}$ and $p_s$ are all
  weakly increasing.

  Using the notation~\eqref{eq:15}, introduce the set
  \begin{equation}
    \label{eq:9}
    E
    :=
    \left\{x \in \reali \colon
      p \mbox{ is differentiable at }x \mbox{ and } p' (x) \neq 0
    \right\} \,.
  \end{equation}

  \paragraph{Proof of~\ref{thm:2:i}.} Fix $x,y \in \reali$ with
  $x < y$. Apply Proposition~\ref{prop:4} with $a = 0$, $b = T$ and as
  $\alpha$, respectively $\beta$, the minimal, respectively maximal,
  backward characteristic from $(T,x)$, respectively $(T,y)$. Note
  that our choice $u \in \C0 ([0,T]; \L1 (\reali; \reali))$ in
  Definition~\ref{def:HCL} allows to select $t_1 = 0$ and $t_2 = T$.
  Then, with the notation~\eqref{eq:15},
  \begin{eqnarray*}
    \int_{p (x)}^{p (y)} u_o (\xi) \d\xi
    & =
    & \int_x^y w (\xi) \d\xi
      -
      \int_0^t \left(
      f\left(u (t, \alpha (t)-)\right)
      -
      \dot\alpha (t) \, u \left(t, \alpha (t)\right)
      \right) \d{t}
    \\
    &
    & +
      \int_0^t \left(
      f\left(u (t, \beta (t)-)\right)
      -
      \dot\beta (t) \, u \left(t, \beta (t)\right)
      \right) \d{t}
    \\
    & =
    & \int_x^y w (\xi) \d\xi
      \qquad\qquad\qquad\qquad\qquad\qquad\qquad
      \mbox{[by~\cite[Theorem~11.1.3]{DafermosBook}]}
    \\
    &
    & - T\left(f\left(w (x)\right) - w (x) \, f'\left(w (x)\right)\right)
      + T\left(f\left(w (y)\right) - w (y) \, f'\left(w (y)\right)\right) \,,
  \end{eqnarray*}
  and similar equalities hold in the case $y<x$.  Let now $x \in E$,
  so that $p$, and hence $w$, are differentiable at $x$.  In the
  previous equality, divide by $y-x$ and pass to the limit as
  $y \to x$. Then, the right hand side converges and, hence, also the
  left hand side. We thus obtain
  \begin{equation}
    \label{eq:3}
    \lim_{y \to x}
    \dfrac{1}{y-x} \int_{p (x)}^{p (y)} u_o (\xi) \d\xi
    =
    w (x)
    - T \, w (x) \; w' (x) \; f''\left(w (x)\right)
    =
    w (x) \; p' (x) \,.
  \end{equation}
  Since $p$ is differentiable at $x$, $p' (x) \neq 0$ and
  using~\eqref{eq:3}, we prove the existence of the following limit,
  at the same time computing its value:
  \begin{eqnarray*}
    \lim_{y\to x}
    \dfrac{1}{p (y)- p (x)}
    \int_{p (x)}^{p (y)} u_o (\xi) \d\xi
    & =
    &  \lim_{y\to x}
      \dfrac{\displaystyle \dfrac{1}{y-x} \int_{p (x)}^{p (y)} u_o (\xi) \d\xi}{\displaystyle \dfrac{p (y)- p (x)}{y-x}}
    \\
    & =
    & \dfrac{\displaystyle \lim_{y\to x} \dfrac{1}{y-x} \int_{p (x)}^{p (y)} u_o (\xi) \d\xi}{\displaystyle \lim_{y\to x}\dfrac{p (y)- p (x)}{y-x}}
    \\
    & =
    & \dfrac{w (x) \; p' (x)}{p' (x) }
    \\
    & =
    & w (x)
  \end{eqnarray*}
  which completes the proof of~\eqref{eq:11}.

  \paragraph{Proof of~\ref{thm:2:ii}.} Let $x \in \reali$ be such that
  $w (x-) > w (x+)$ and $v \in \left] w (x+), w (x-) \right]$.
  Introduce the minimal backward characteristic
  $\alpha (t) = x- (T-t) \, f'\left(w (x-)\right)$ and the line
  $\beta (t) = x - (T-t) \, f' (v)$. Apply Proposition~\ref{prop:4}
  with $t_1 = 0$, $t_2 = T$ to obtain
  \begin{eqnarray*}
    - \int_{\alpha (0)}^{\beta (0)} u (0,x) \d{x}
    & =
    & \int_0^T \left(
      f\left(u \left(t, \alpha (t)-\right)\right)
      -
      \dot\alpha (t) \, u\left(t, \alpha (t)-\right)
      \right) \d{t}
    \\
    &
    & - \int_0^T \left(
      f\left(u (t, \beta(t)+\right)
      -
      \dot \beta (t) \, u\left(t, \beta (t)+\right)
      \right) \d{t} \,.
  \end{eqnarray*}
  To compute the first summand in the right hand side recall that $u$
  is constant along minimal backward characteristics, while the
  convexity of $f$ ensures that
  $f (w) - f' (v) \, w \geq f (v) - f' (v) \, v$. We thus have
  \begin{equation}
    \label{eq:23}
    \int_{\alpha (0)}^{\beta (0)} u (0,x) \d{x}
    \geq
    T \, \left(v \, f'(v)-f(v)\right)
    -
    T \, \left(w(x-) \, f'\left(w(x-)\right)-f\left(w(x-)\right)\right) \,,
  \end{equation}
  proving the latter inequality in~\ref{thm:2:ii}. The proof of the
  former one is entirely analogous.
\end{proofof}

\begin{proofof}{Lemma~\ref{lem:equivalence}}
  The equivalence between~\ref{thm:2:i} in Theorem~\ref{thm:2}
  and~\ref{thm:3:1} in Theorem~\ref{thm:3} in Theorem~\ref{thm:3} is
  immediate, thanks to~ the relation $u = \partial_x U$.

  To prove the equivalence between~\ref{thm:2:ii} in
  Theorem~\ref{thm:2} and~\ref{thm:3:2} in Theorem~\ref{thm:3}, note
  that
  \begin{displaymath}
    \begin{array}{ccccc}
      \Pi_x
      & \colon
      & \left[ w (x+), w (x-) \right]
      & \to
      & \left[p (x-), p (x+) \right]
      \\
      &
      &v
      & \to
      & x - T \; f' (v)
    \end{array}
  \end{displaymath}
  is a bijective map, since
  \begin{displaymath}
    y = \Pi_x (v)
    \iff
    v = g\left(\dfrac{x-y}{T}\right) \,.
  \end{displaymath}
  Hence, straightforward computations yield
  \begin{eqnarray*}
    v \, f' (v) - f (v)
    & =
    & g\left(\dfrac{x-y}{T}\right) \;
      f'\left(g\left(\dfrac{x-y}{T}\right)\right)
      - f\left(g\left(\dfrac{x-y}{T}\right)\right)
    \\
    & =
    & \dfrac{x-y}{T} \;
      g\left(\dfrac{x-y}{T}\right)
      - f\left(g\left(\dfrac{x-y}{T}\right)\right)
    \\
    & =
    & f^*\left(\dfrac{x-y}{T}\right) \,.
  \end{eqnarray*}
  Using the above equality at points $x$ of jump in $p$, the
  equivalence \ref{thm:2:ii}~$\iff$~\ref{thm:3:2} follows.
\end{proofof}

\begin{lemma}
  \label{lem:2}
  Let~\eqref{eq:H} hold and let $s$ be defined as in~\eqref{eq:7}.  If
  \begin{equation}
    \label{eq:26}
    \left\{x_1 < x_2 \; \mbox{ and } \; y_1 \leq y_2\right\}
    \qquad \mbox{ or } \qquad
    \left\{x_1 \leq x_2 \; \mbox{ and } \; y_1 < y_2\right\}
  \end{equation}
  then
  \begin{equation}
    \label{eq:27}
    s(T, x_1, y_1) + s(T, x_2, y_2) < s(T, x_1, y_2) + s(T, x_2, y_1) \,.
  \end{equation}
\end{lemma}

\begin{proof}
  Define $A := \frac{x_1-y_2}{T}$, $B := \frac{x_2-y_1}{T}$,
  $C := \frac{x_1-y_1}{T}$ and $D := \frac{x_2-y_2}{T}$.
  Using~\eqref{eq:26} we immediately obtain $A + B = C +D$, $A < D$,
  $A < C$, $C < B$ and $D < B$.  So we can conclude that there exists
  $\theta\in \left]0,1\right[$ such that
  $C = \theta \, A + (1-\theta) \, B$ and
  $D = (1-\theta) \, A + \theta \, B$.

  Denote by $\Delta$ the difference between the right side and left
  side of~\eqref{eq:27}. By~\eqref{eq:7} we get
  \begin{eqnarray*}
    \Delta
    & =
    & T\left(
      f^*\left(\frac{x_1-y_2}{T}\right)
      +
      f^*\left(\frac{x_2-y_1}{T}\right)
      -
      f^*\left(\frac{x_1-y_1}{T}\right)
      -
      f^*\left(\frac{x_2-y_2}{T}\right)
      \right)
    \\
    & =
    & T\left(
      f^*(A)
      +
      f^*(B)
      -
      f^*(C)
      -
      f^*(D)
      \right)
  \end{eqnarray*}
  The strict convexity of $f^*$ now ensures that $\Delta>0$,
  completing the proof.
\end{proof}

\begin{proofof}{Theorem~\ref{thm:3}, (sufficiency)}
  Let $U_o$ be such that conditions~\ref{thm:3:1} and~\ref{thm:3:2}
  hold. Then, we prove that the solution $u$ to~\eqref{eq:35} with
  $u (0) = \partial_x U_o$ also satisfies $u (T) = w$.

  We use below Lax-Hopf Formula, i.e., Proposition~\ref{prop:1}. To
  this aim, define $p$ as in~\eqref{eq:15} and the Legendre transform
  $f^*$ of $f$ as in~\eqref{eq:7}. It is sufficient to show that for
  a.e.~$x \in \reali$ and for all $y \in \reali$
  \begin{displaymath}
    T \, f^*\left(\frac{x-p (x)}{T}\right) + U_o\left(p (x)\right)
    \leq
    T \, f^*\left(\frac{x-y}{T}\right) + U_o(y) \,,
  \end{displaymath}
  which, by~\eqref{eq:7}, is equivalent to
  \begin{equation}
    \label{eq:28}
    \forall y \in \reali\,, \qquad
    s \left(T, x, p (x)\right) \leq s (T, x, y) \,.
  \end{equation}

  \paragraph{Step~1:} Let ${\bar x} \in \reali$ be a point where $p$
  is differentiable. Then, for all $x \in \reali$, the map
  $\xi \to s\left(T, x, p (\xi)\right)$ is differentiable at
  ${\bar x}$ and we have
  \begin{equation}
    \label{eq:8}
    \forall x
    \leq
    {\bar x}, \quad
    \dfrac{\d~}{\d{{\bar x}}} \, s\left(T,x,p({\bar x})\right)
    \geq
    0 \,,
    \quad \mbox{ and } \quad
    \forall x
    \geq
    {\bar x}, \quad
    \dfrac{\d~}{\d{{\bar x}}} \, s\left(T,x,p({\bar x})\right)
    \leq
    0 \,.
  \end{equation}
  \paragraph{Proof of Step~1:}
  Consider first the case $p' ({\bar x}) \neq 0$, so that
  $p ({\bar x}) > 0$ since $p$ is weakly increasing. Then, using the
  definition of $s$ in~\eqref{eq:7}, hypothesis~\ref{thm:3:1} and the
  regularity of $f^*$
  \begin{eqnarray*}
    &
    & \dfrac{s\left(T, x, p (\xi)\right) - s\left(T, x, p ({\bar x})\right)}{\xi-{\bar x}}
    \\
    & =
    & T \,
      \dfrac{f^*\left(\frac{x-p (\xi)}{T}\right) - f^*\left(\frac{x-p ({\bar x})}{T}\right)}%
      {\xi-{\bar x}}
      +
      \dfrac{U_o\left(p (\xi)\right) - U_o\left(p ({\bar x})\right)}{\xi-{\bar x}}
    \\
    & =
    & T \,
      \dfrac{f^*\left(\frac{x-p (\xi)}{T}\right) - f^*\left(\frac{x-p ({\bar x})}{T}\right)}%
      {\xi-{\bar x}}
      +
      \dfrac{p (\xi) - p ({\bar x})}{\xi -{\bar x}}
      \;
      \dfrac{U_o\left(p (\xi)\right) - U_o\left(p ({\bar x})\right)}{p (\xi) - p ({\bar x})}
    \\
    & {\to \atop\xi \to {\bar x}}
    & - p' ({\bar x}) \; g\left(\frac{x-p ({\bar x})}{T}\right)
      +
      p' ({\bar x}) \; w ({\bar x})
    \\
    & =
    & p'({\bar x}) \left(
      g\left(\frac{{\bar x}-p({\bar x})}{T}\right)
      -
      g\left(\frac{x-p({\bar x})}{T}\right)
      \right) \,,
  \end{eqnarray*}
  because
  \begin{displaymath}
    w({\bar x})
    =
    g\left(\frac{{\bar x}-p({\bar x})}{T}\right)
  \end{displaymath}
  and since $g$ is increasing and $p'({\bar x}) > 0$, the present
  claim is proved in the case $p' ({\bar x}) \neq 0$.

  Consider now the case $p' ({\bar x}) = 0$ and follow computations
  similar to the ones above:
  \begin{eqnarray*}
    &
    & \modulo{
      \dfrac{s\left(T, x, p (\xi)\right) - s\left(T, x, p ({\bar x})\right)}{\xi-{\bar x}}
      -
      T \,
      \dfrac{f^*\left(\frac{x-p (\xi)}{T}\right) - f^*\left(\frac{x-p ({\bar x})}{T}\right)}{\xi-{\bar x}}
      }
    \\
    & =
    & \modulo{
      \dfrac{U_o\left(p (\xi)\right) - U_o\left(p ({\bar x})\right)}{\xi-{\bar x}}
      }
    \\
    & \leq
    & \norma{u_o}_{\L\infty (\reali; \reali)} \;
      \modulo{\dfrac{p (\xi) - p ({\bar x})}{\xi-{\bar x}}}
    \\
    & {\to \atop\xi \to {\bar x}}
    & 0 \,.
  \end{eqnarray*}
  The proof of the present claim is completed, since
  $\lim_{\xi \to {\bar x}} \dfrac{f^*\left(\frac{x-p (\xi)}{T}\right)
    - f^*\left(\frac{x-p (x)}{T}\right)}{\xi-{\bar x}} = 0$.

  \paragraph{Step~2:} Let $\bar x$ be a point of jump of
  $p$. Since~\ref{thm:3:2} holds, we have that
  \begin{displaymath}
    \begin{array}{r@{\;}c@{\;}l@{\qquad}r@{\;}c@{\;}l@{\qquad}r@{\;}c@{\;}l}
      \forall x
      & \geq
      & \bar x ,
      & \forall y
      & \in
      & [p(\bar x-),p(\bar x+)],
      & s\left(T,x,p(\bar x+)\right)
      & \leq
      & s(T,x,y) \,,
      \\
      \forall x
      & \leq
      & \bar x ,
      & \forall y
      & \in
      & [p(\bar x-),p(\bar x+)],
      & s\left(T,x,p(\bar x-)\right)
      & \leq
      & s(T,x,y) \,.
    \end{array}
  \end{displaymath}

  \paragraph{Proof of Step~2:}
  Consider the relations in~\ref{thm:3:2}, which can be rewritten as
  \begin{displaymath}
    \forall y \in [p(\bar x-),p(\bar x+)],
    \qquad
    s\left(T, \bar x, p (\bar x\pm)\right)
    -
    s (T, \bar x, y)
    \leq
    0
  \end{displaymath}
  Apply Lemma~\ref{lem:2} in the two cases
  \begin{displaymath}
    \begin{array}{r@{\,}c@{\,}l@{\qquad}r@{\,}c@{\,}l@{\qquad}r@{\,}c@{\,}l@{\qquad}r@{\,}c@{\,}l@{\qquad}r@{\,}c@{\,}l@{\qquad}r@{\,}c@{\,}l}
      x
      & >
      & \bar x
      & x_1
      & =
      & x
      & x_2
      & =
      & \bar x
      & y_1
      & =
      & p (\bar x+)
      & y_2
      & =
      & y
      \\
      x
      & <
      & \bar x
      & x_1
      & =
      & x
      & x_2
      & =
      & x
      & y_1
      & =
      & p (\bar x-)
      & y_2
      & =
      & y
    \end{array}
  \end{displaymath}
  obtaining
  \begin{displaymath}
    \begin{array}{rclcl}
      s \left(T,x,p (\bar x+)\right)
      & <
        s(T, x, y)
        +
        s \left(T, \bar x, p (\bar x+)\right)
        -
        s (T, \bar x, y)
      & <
      & s(T, x, y)
      \\
      s \left(T,x,p (\bar x-)\right)
      & <
        s(T, x, y)
        +
        s \left(T, \bar x, p (\bar x-)\right)
        -
        s (T, \bar x, y)
      & <
      & s(T, x, y)
    \end{array}
  \end{displaymath}
  completing the proof of~\textbf{Step~2}..

  \paragraph{Step~3:} For any $x \in \reali$, the map
  $\xi \to s\left(T,x, p (\xi)\right)$ attains its minimum at
  $x$. Equivalently
  \begin{equation}
    \label{eq:24}
    \forall y \in \overline{p (\reali)} \qquad
    s (T, x, y) \geq s \left(T, x, p (x)\right) \,.
  \end{equation}

  \paragraph{Proof of Step~3:} Fix $x \in \reali$ and denote
  $S (\xi) = s\left(T, x, p (\xi)\right)$.

  Consider first the set $\left[x, +\infty\right[$. $p$ is weakly
  increasing on $\left[x, +\infty\right[$, hence it is differentiable
  a.e.~on $\left[x, +\infty\right[$ . Then, by~\textbf{Step~1}, $S$ is
  differentiable a.e.~and $S' (\xi) \geq 0$ for
  a.e.~$\xi \in \left[x, +\infty\right[$. At all jump points
  $\xi \in \left[x, +\infty\right[$, by~\textbf{Step~2}, we have
  $S (\xi-) \leq S (\xi+)$. Since the map $y \to s (T, x, y)$ is
  Lipschitz and $p$ is in $\SBV (\reali, \reali)$,
  by~\cite[Proposition~1.2]{Ambrosio1995}, we have that also
  $S \in \SBV (\reali; \reali)$. Thus, $S$ is weakly increasing on
  $\left[x, +\infty\right[$.

  An entirely symmetric argument shows that $S$ is weakly decreasing
  on $\left]-\infty, x\right]$, completing the proof
  of~\textbf{Step~3}.

  \paragraph{Step~4:} For any $x \in \reali$,
  \begin{equation}
    \label{eq:29}
    \forall y \in \reali \setminus \overline{p (\reali)} \qquad
    s (T, x, y) \geq s \left(T, x, p (x)\right) \,.
  \end{equation}

  \paragraph{Proof of Step~4:} Fix $x \in \reali$ and
  $y \in \reali \setminus \overline{p (\reali)}$. By~\eqref{eq:15} and
  since $w \in \L\infty (\reali; \reali)$,
  $\lim_{\xi \to \pm\infty} p (\xi) = \pm \infty$. Hence, we can
  define
  $\bar x = \sup \left\{\xi \in \reali \colon p (\xi) < y\right\}$ so
  that, thanks to the fact that $p$ is weakly increasing,
  $y \in \left] p (\bar x-), p (\bar x+)\right[$. Then,
  by~\textbf{Step~2},
  \begin{displaymath}
    \begin{array}{r@{\,}c@{\,}l@{\quad\Rightarrow\quad}r@{\,}c@{\,}l}
      \bar x
      & \geq
      & x
      & s (T, x, y)
      & \geq
      & s\left(T, x, p (\bar x-)\right) \,,
      \\[3pt]
      \bar x
      & \leq
      & x
      & s (T, x, y)
      & \geq
      & s\left(T, x, p (\bar x+)\right) \,,
    \end{array}
  \end{displaymath}
  so that, using~\textbf{Step~3}, the proof of~\textbf{Step~4} is
  completed.

  \smallskip

  The proof is completed, since~\eqref{eq:28} follows
  from~\eqref{eq:24} and~\eqref{eq:29}.
\end{proofof}

\begin{proofof}{Theorem~\ref{thm:4}}
  The proof is divided into the following steps.

  \paragraph{If $U_o \in \II_T (W)$ then~\ref{thm:4:1}
    and~\ref{thm:4:2} hold.}

  If $U_o \in \II_T (W)$ then,
  by~\cite[Theorem~1.1]{KarlsenRisebro2002},
  $\partial_x U_o \in \ii_T (\partial_x W)$. Theorem~\ref{thm:3}
  directly ensures that~\ref{thm:4:1} holds. Concerning~\ref{thm:4:2},
  consider a point $\bar x$ of jump for $p$ and choose
  $y = p (\bar x-)$, respectively $y = p (\bar x+)$, in the first,
  respectively second, inequality in~\ref{thm:3:2}, we obtain the
  equality
  \begin{displaymath}
    U_o
    \left(
      p (\bar x-)\right) + T \, f^*\left(\frac{\bar x-p (\bar x-)}{T}
    \right)
    =
    U_o
    \left(
      p (\bar x+)\right) + T \, f^*\left(\frac{\bar x-p (\bar x+)}{T}
    \right) \,.
  \end{displaymath}
  Call $Q$ this value and get
  $U_o (y) + T \, f^*\left(\frac{\bar x-y}{T}\right) \geq Q$. We are
  left to prove that $Q = W(\bar x)$.

  Call $U(t,x) := (\SS_t U_o) (x)$ and introduce
  $u_o := \partial_x U_o$, $w := \partial_x W$ and
  $u (t,x) := (\ss_t u_o)(x)$. Let $\gamma$ be the minimal backward
  characteristic emanating from $(T, \bar x)$, so that
  $\gamma (t) = \bar x + (t-T) \, f'\left(w (\bar x-)\right)$ and
  $u \left(t, \gamma (t)-\right) = w (\bar x-)$ by~\cite[Theorem~3.2
  and Theorem~3.3]{Dafermos77}. Clearly, by~\eqref{eq:15},
  $p (\bar x-) = \gamma (0)$.

  By Proposition~\ref{prop:sol}
  and~\cite[Theorem~1.1]{KarlsenRisebro2002}, there exists
  $c \in \reali$ such that for $(t,x) \in [0,T] \times \reali$
  \begin{eqnarray*}
    U (t,x)
    & =
    & \int_{\gamma (t)}^x u (t,\xi) \d\xi
      +
      \int_0^t \left(
      \dot\gamma (\tau) \, u\left(\tau,\gamma (\tau)-\right)
      -
      f\left(u\left(\tau, \gamma (\tau)-\right)\right)
      \right)
      \d\tau
      +
      c
    \\
    & =
    & \int_{\gamma (t)}^x u (t,\xi) \d\xi
      +
      t
      \left(
      f'\left(w (\bar x-)\right) \, w (\bar x-)
      -
      f\left(w (\bar x-)\right)
      \right)
      +
      c\,.
  \end{eqnarray*}
  Direct evaluations of the latter expression above yield
  \begin{displaymath}
    W (\bar x)
    =
    U (T, \bar x)
    =
    T \, f^*\left(\frac{\bar x - p (\bar x-)}{T}\right) +c
    \quad \mbox{ and } \quad
    U_o \left(p (\bar x-)\right)
    =
    U\left(0, p (\bar x-)\right)
    =
    c \,.
  \end{displaymath}
  so that $W (\bar x) = Q$, completing the proof of this part.

  \paragraph{If~\ref{thm:4:1} and~\ref{thm:4:2} hold, then
    $U_o \in \II_T (W)$.}

  If~\ref{thm:4:1} holds, then clearly $U_o$
  satisfies~\ref{thm:3:1}. Moreover, by~\ref{thm:4:2}, if $\bar x$ is
  a point of jump of $p$,
  \begin{eqnarray*}
    U_o\left(p (\bar x\pm)\right)
    + T \, f^*\left(\frac{\bar x - p (\bar x\pm)}{T}\right)
    & \leq
    & U_o (y) + T \, f^*\left(\frac{\bar x-y}{T}\right)
    \\
    \dfrac{U_o\left(p (\bar x\pm)\right) - U_o (y)}{T}
    & \leq
    & f^*\left(\frac{\bar x-y}{T}\right)
      -
      f^*\left(\frac{\bar x - p (\bar x\pm)}{T}\right)
  \end{eqnarray*}
  which implies~\ref{thm:3:2}. Then, Theorem~\ref{thm:3} applies and
  ensures that $\partial_x U_o\in \ii_T (\partial_x
  W)$. By~\cite[Theorem~1.1]{KarlsenRisebro2002}, it follows that
  $\partial_x \SS_T U_o = \ss_T \partial_x U_o = \partial_x W$ and,
  hence, there is a constant $c \in \reali$ such that
  $\ss_T U - W = c$. Thus, $\SS_T (U_o + c) = \SS_T U_o + c = W$ and
  $(U_o + c) \in \II_T (W)$ so that $(U_o+c)$ satisfies the equality
  in~\ref{thm:4:2}, as proved in the previous claim. By assumption,
  also $U_o$ satisfies the equality in~\ref{thm:4:2}, hence $c=0$,
  completing the proof.
\end{proofof}

\section{Proofs Related to \S~\ref{subs:Geo}}
\label{sec:proof:Geo}

\begin{proofof}{Proposition~\ref{prop:top}}
  We prove the different parts separately.

  \paragraph{Proof of~\ref{prop:top:1}:}The strong $\L1$ closure of
  $\ii_T (w)$ directly follows from the strong $\L1$ continuity of the
  semigroup generated by~\eqref{eq:1}, see~\cite[Chapter~6,
  \S~4]{BressanLectureNotes}.

  \paragraph{Proof of~\ref{prop:top:2}:}To prove that $\ii_T (w)$ has
  empty interior, fix $u_o$ in $\ii_T (w)$ and use the
  characterization of $\ii_T (w)$ provided by
  Theorem~\ref{thm:3}. Note that by Proposition~\ref{prop:D}, there
  exists an $\bar x \in \reali$ such that either~\ref{thm:3:1}
  or~\ref{thm:3:2} in Theorem~\ref{thm:3} holds.

  \subparagraph{Let~\ref{thm:3:1} hold at a given
    $\bar x \in \reali$.}  Define the sequence of initial data
  \begin{displaymath}
    u_o^n (x)
    :=
    u_o (x) + \caratt{]p(\bar x) - 1/n, p (\bar x) + 1/n[} (x) \,.
  \end{displaymath}
  Clearly, $u_o^n \to u_o$ strongly in $\L1$ as $n \to +\infty$. On
  the other hand, with reference to~\eqref{eq:18} and choosing $U_o^n$
  so that $\partial_x U_o^n = u_o^n$, we have that for $y$
  sufficiently near to $x$,
  \begin{displaymath}
    \dfrac{U_o^n\left(p (y)\right) - U_o^n\left(p (\bar x)\right)}%
    {p (y) - p (\bar x)}
    =
    \dfrac{U_o\left(p (y)\right) - U_o\left(p (\bar x)\right)}%
    {p (y) - p (\bar x)}
    +
    1
    \underset{y \to \bar x}{\to}
    w (\bar x) + 1\,,
  \end{displaymath}
  showing that $u_o^n \not \in \ii_T (w)$ by Theorem~\ref{thm:3}.

  \subparagraph{Let~\ref{thm:3:2} hold at a given
    $\bar x \in \reali$.}  Define the sequence of initial data
  \begin{displaymath}
    u_o^n (x)
    :=
    u_o (x) - C \; \caratt{[p (\bar x-), p (\bar x-) + 1/n]} (x) \,,
  \end{displaymath}
  for a sufficiently large constant $C$ that is explicitly chosen
  in~\eqref{eq:34}. We have that for
  $y \in \left] p (\bar x-), p (\bar x-) + \frac{1}{n}\right[$,
  \begin{eqnarray*}
    &
    & \dfrac{U^n_o (y) - U^n_o\left(p (\bar x-)\right)}{T}
      -
      f^*\left(\dfrac{\bar x-p (\bar x-)}{T}\right)
      +
      f^*\left(\dfrac{\bar x - y}{T}\right)
    \\
    & =
    & \dfrac{U_o (y) - U_o\left(p (\bar x-)\right)}{T}
      -
      f^*\left(\dfrac{\bar x-p (\bar x-)}{T}\right)
      +
      f^*\left(\dfrac{\bar x - y}{T}\right)
      -
      C \, \dfrac{y-p (\bar x-)}{T}
    \\
    & \leq
    & \left(\norma{u_o}_{\L\infty (\reali; \reali)} - C\right)
      \dfrac{y-p (\bar x)}{T}
      +
      f^*\left(\dfrac{\bar x - y}{T}\right)
      -
      f^*\left(\dfrac{\bar x-p (\bar x-)}{T}\right)
    \\
    & \leq
    & \left(
      \norma{u_o}_{\L\infty (\reali; \reali)}
      -
      g\left(\frac{\bar x - p (\bar x-)}{T}\right)
      -
      C
      \right)
      \dfrac{y-p (\bar x)}{T}
      +
      o \left(y - p (\bar x-)\right)
      \qquad \mbox{ as } y \to p (\bar x -) \,,
  \end{eqnarray*}
  so that as soon as $C$ is chosen satisfying
  \begin{equation}
    \label{eq:34}
    C > \norma{u_o}_{\L\infty (\reali; \reali)}
    -
    g\left(\frac{\bar x - p (\bar x-)}{T}\right)
  \end{equation}
  we have
  \begin{equation}
    \label{eq:30}
    \dfrac{U^n_o (y) - U^n_o\left(p (\bar x-)\right)}{T} -
    f^*\left(\dfrac{\bar x-p (\bar x-)}{T}\right) + f^*\left(\dfrac{\bar
        x - y}{T}\right)
    <
    0
  \end{equation}
  for all $y$ sufficiently near to and larger than $p (\bar x-)$.
  But~\eqref{eq:30} contradicts~\ref{thm:3:2} in Theorem~\ref{thm:3},
  so we obtain that $u_o^n \not\in \ii (w)$, although $u^n_o \to u_o$
  in $\L1 (\reali;\reali)$ and $u_o \in \ii (w)$.
\end{proofof}

The following remark provides a basic linear algebra observation of
use in the subsequent proof of Proposition~\ref{prop:geo}.

\begin{remark}
  \label{rem:algebra}
  Let $V$ be a vector spaces,
  $L_\alpha, \Lambda_\alpha \colon V \to \reali$ be linear maps and
  $m_\alpha, \mu_\alpha$ be real numbers, with $\alpha$ varying in a
  suitable set of indices $I$. Assume there exists a unique
  $v^* \in V$ such that for all $\alpha \in I$,
  $L_\alpha v^* = m_\alpha$ and $\Lambda_\alpha v^* = \mu_\alpha$.
  Then, the set
  $\{v \in V \colon \forall \alpha \in I \quad L_\alpha v = m_\alpha
  \mbox{ and } \Lambda_\alpha v \geq \mu_\alpha\}$ is a cone with
  vertex at $v^*$, which is its unique extremal point.
\end{remark}

\begin{proofof}{Proposition~\ref{prop:geo}}
  We split the proof in different steps.

  \paragraph{Proof of~\ref{prop:geo:1}:} Consider the two implications
  separately.

  \subparagraph{If $w \in \C0 (\reali; \reali)$, then $\ii_T (w)$ is a
    singleton.} Let $w \in \C0 (\reali; \reali)$. Then, the set
  $X_{ii}$ in~\eqref{eq:22} is empty. By Proposition~\ref{prop:D},
  $\mathcal{L} (\reali \setminus X_i) =0$.  For any
  $u_o \in \ii_T (w)$ and any $\bar x$ in $X_i$,
  \begin{displaymath}
    \begin{array}{r@{\,}c@{\,}l@{\qquad}l}
      w (\bar x)
      & =
      & \displaystyle
        \lim_{y \to \bar x} \dfrac{1}{p (y) - p (\bar x)} \,
        \int_{p (\bar x)}^{p (y)} u_o (\xi) \d\xi
      & \mbox{[by Theorem~\ref{thm:2}]}
      \\
      & =
      & \displaystyle
        \lim_{r \to 0} \dfrac{1}{r} \, \int_{p(\bar x)}^{p(\bar x) + r} u_o (\xi) \d\xi
      & \mbox{[since } p' (\bar x) > 0
        \mbox{ and } p\in \C0 (\reali; \reali)\mbox{]}
      \\
      & =
      & u_o\left(p (\bar x)\right)
      & \mbox{[for the precise representative~\eqref{eq:37} of } u_o \mbox{]}
    \end{array}
  \end{displaymath}
  Indeed, we used above the fact that if $p' (\bar x) > 0$ and
  $p\in \C0 (\reali; \reali)$, then for all $r>0$ sufficiently small,
  there exists a $y_r$ such that $r = p (y_r) - p (\bar x)$ and
  $y_r \to \bar x$ as $r\to 0$.

  \subparagraph{If $w$ admits a point of discontinuity $\bar x$, then
    $\ii_T (w)$ is not a singleton. } A first element of $\ii_T (w)$
  is the map $u_o^{*}$ defined in Theorem~\ref{thm:1}. A second map
  can be constructed prolonging the shock at $\bar x$ backward to
  $0$. To this aim, we define
  \begin{displaymath}
    u^{\sharp}_o (x)
    :=
    \left\{
      \begin{array}{lrcl}
        u^*_o (x)
        & x
        & \leq
        & p (\bar x-)
        \\
        w (\bar x-)
        & x
        & \in
        & \left] p (\bar x-) , x_\sharp \right]
        \\
        w (\bar x+)
        & x
        & \in
        & \left] x_\sharp, p (\bar x+) \right]
        \\
        u^*_o (x)
        & x
        & >
        & p (\bar x+)
      \end{array}
    \right.
    \quad \mbox{ where } \quad
    \begin{array}{r@{\;}c@{\;}l}
      \lambda^\sharp
      & =
      & \dfrac{f \left(w (\bar x+)\right) - f\left(w (\bar x -)\right)}%
        {w (\bar x+) - w (\bar x-)}
      \\
      x_\sharp
      & =
      & \bar x
        -
        \lambda^\sharp \, T \,.
    \end{array}
  \end{displaymath}
  We check that $u_o^\sharp$ satisfies~\ref{thm:2:ii} in
  Theorem~\ref{thm:2} at $\bar x$. To this aim, set
  $v^\sharp = g (\lambda^\sharp)$ and compute
  \begin{displaymath}
    \frac{1}{T} \int_{x-T f'(w(x-))}^{x-T f'(v)}{u_o(\xi)\d{\xi}}
    =
    \left\{
      \begin{array}{lr@{\,}c@{\,}l}
        \left(f'\left(w (\bar x-)\right) - f' (v)\right) w (\bar x-)
        & v
        & <
        & v^\sharp
        \\[10pt]
        \begin{array}{@{}l@{}}
          \left(f'\left(w (\bar x-)\right) - f' (v^\sharp)\right) w (\bar x-)
          \\
          \quad +
          \left(f' (v^\sharp) - f'\left(v\right) \right) w (\bar x+)
        \end{array}
        & v
        & >
        & v^\sharp
      \end{array}
    \right.
  \end{displaymath}
  so that, with reference to~\ref{thm:2:ii} in Theorem~\ref{thm:2},
  denote
  \begin{displaymath}
    \Delta
    :=
    \frac{1}{T} \int_{\bar x-T f'(w(\bar x-))}^{\bar x-T f'(v)}{u_o(\xi)\d{\xi}}
    -
    \left(w(\bar x-) \, f'\left(w(\bar x-)\right)- f\left(w(\bar x-)\right)\right)
    +
    \left(v \, f'(v)-f(v)\right)
  \end{displaymath}
  and, for $v < v^\sharp$ by the convexity of $f$ we obtain
  \begin{displaymath}
    \Delta
    =
    f\left(w(x-)\right)
    - f(v)
    - f' (v) \left(w (\bar x-) - v\right)
    \geq 0 \,.
  \end{displaymath}
  If $v > v^\sharp$ , we use the identity
  $f' (v^\sharp) \left(w (\bar x+) - w (\bar x-)\right) = f \left(w
    (\bar x+)\right) - f\left(w (\bar x -)\right)$ to obtain
  \begin{eqnarray*}
    \Delta
    & =
    & f\left(w (\bar x+)\right)
      - f\left(w (\bar x -)\right)
      + f'\left(w (\bar x-)\right) \, w (\bar x-)
      - f' (v) \, w (\bar x +)
    \\
    &
    & \quad
      - f'\left(w (\bar x-)\right) \, w (\bar x-)
      + f\left(w (\bar x)\right)
      + f' (v) \, v
      - f (v)
    \\
    & =
    & f\left(w (\bar x+)\right)
      - f (v)
      - f' (v) \left(w (\bar x+) - v\right)
    \\
    & \geq
    & 0
  \end{eqnarray*}
  where we used again the convexity of $f$.

  Entirely similar computations apply to the second inequality
  in~\ref{thm:2:ii}. Hence $u_o^\sharp \in \ii_T (w)$ and, since
  $u_o^\sharp \neq u_o^*$, the proof of~\ref{prop:geo:1} is completed.

  \paragraph{Proof of~\ref{prop:geo:2}:} The convexity of $\ii_T (w)$
  directly follows from Lax-Hopf formula~\cite[Theorem~2.1]{Lax}, as
  well as from~\cite[Theorem~6.2]{GosseZuazua} or by direct inspection
  of the conditions provided by Theorem~\ref{thm:2} or
  Theorem~\ref{thm:3}.

  \subparagraph{$\ii_T (w)$ is a cone with $u_o^*$ at its vertex.} (We
  follow the abstract reasoning sketched in Remark~\ref{rem:algebra}.)
  Let $u_o^*$ be as defined in Lemma~\ref{lem:ExUn}. For a
  $u_o \in \ii_T (w) \setminus \{u_o^*\}$, and for all
  $\theta \in \reali^+$, define
  $u_o^\theta := u_o^* + \theta \, (u_o - u_o^*)$ and, passing to
  primitives, $U_o^\theta = U_o^* + \theta \, (U_o - U_o^*)$.

  If $\bar x$ is such that $p$ is differentiable at $\bar x$ and
  $p' (\bar x) > 0$, then by~\ref{thm:3:1} in Theorem~\ref{thm:3}
  and~\ref{lem:ex:1} in Lemma~\ref{lem:ExUn},
  \begin{eqnarray*}
    &
    & \lim_{x \to \bar x}
      \dfrac{U_o^\theta\left(p (x)\right) - U_o^\theta\left(p (\bar x)\right)}%
      {p (\bar x) - p (x)}
    \\
    & =
    & \lim_{x \to \bar x}
      \dfrac{U_o^*\left(p (x)\right) - U_o^*\left(p (\bar x)\right)}%
      {p (\bar x) - p (x)}
      +
      \lim_{x \to \bar x} \theta
      \left(
      \dfrac{U_o\left(p (x)\right) - U_o\left(p (\bar x)\right)}%
      {p (\bar x) - p (x)}
      -
      \dfrac{U_o^*\left(p (x)\right) - U_o^*\left(p (\bar x)\right)}%
      {p (\bar x) - p (x)}
      \right)
    \\
    & =
    & w (x) \,,
  \end{eqnarray*}
  proving that also $U_o^\theta$ satisfies~\ref{thm:3:1} in
  Theorem~\ref{thm:3}.

  Choose now $x \in \reali$ such that $p (x-) < p (x+)$ and compute:
  \begin{eqnarray*}
    &
    & \dfrac{U_o^\theta \left(p (x+)\right) - U_o^\theta(y)}{T}
    \\
    & =
    & \dfrac{U_o^\theta \left(p (x+)\right) - U_o^\theta(y)}{T}
      +
      \theta
      \left(
      \dfrac{U_o \left(p (x+)\right) - U_o(y)}{T}-
      \dfrac{U_o^\theta \left(p (x+)\right) - U_o^\theta(y)}{T}
      \right)
    \\
    & \leq
    & f^*\left(\dfrac{x-y}{T}\right)
      -
      f^*\left(\dfrac{x-p (x+)}{T}\right) .
  \end{eqnarray*}
  An entirely similar computations applies to
  $\frac{U_o^\theta(y) - U_o^\theta \left(p (x+)\right)}{T}$.  We thus
  proved that also $U_o^\theta$ satisfies~\ref{thm:3:1}
  and~\ref{thm:3:2} in Theorem~\ref{thm:3}, hence it belongs to
  $\ii_t (w)$.

  \subparagraph{If $u_o \in \ii_T (w)$ is different from $u_o^*$, then
    it is not an extremal point of $\ii_T (w)$.} This statement
  directly follows from~\ref{prop:geo:3}, which we prove independently
  below.

  \paragraph{Proof of~\ref{prop:geo:3}:}

  Preliminary, note that~\ref{thm:3:2} in Theorem~\ref{thm:3} at $x$
  is equivalent to require that for all $x \in \reali$ such that
  $w (x-) \neq w (x+)$,
  \begin{equation}
    \label{eq:32}
    \begin{array}{@{}r@{\,}c@{\,}l@{}}
      U_o\left(p (x+)\right) + T f^*\left(\dfrac{x-p (x+)}{T}\right)
      & =
      &  U_o\left(p (x-)\right) + T f^*\left(\dfrac{x-p (x-)}{T}\right)
      \\
      \forall y \in \left] p (x-), p (x+) \right[ \qquad
      U_o(y) + T f^*\left(\dfrac{x-y}{T}\right)
      & \geq
      &  U_o\left(p (x-)\right) + T f^*\left(\dfrac{x-p (x-)}{T}\right) \,.
    \end{array}
  \end{equation}

  Fix $N$ and $u_o \in \ii_T (w)$ as in the statement
  of~\ref{prop:geo:3}. Let $U_o$ be a primitive of
  $u_o$. By~\ref{prop:geo:1}, we may assume that there exists an
  $x \in \reali$ such that~\eqref{eq:32} applies. If $u_o \neq u_o^*$,
  then for a suitable $x \in \reali$ and
  $\bar y \in \left] p (x-), p (x+) \right[$, the strict inequality
  has to hold in the latter relation above computed at $\bar
  y$. Hence, there exist positive $\eta$ and $\epsilon$ such that
  \begin{equation}
    \label{eq:17}
    \!\!
    \forall y \in \left]\bar y - \eta, \bar y + \eta \right[
    \quad
    \left\{
      \begin{array}{l@{}}
        y \in \left]p (x-), p (x+) \right[
        \mbox{, \quad and}
        \\
        U_o(y) + T f^* \! \left(\dfrac{x-y}{T}\right)
        \geq
        U_o\left(p (x-)\right) + T f^* \! \left(\dfrac{x-p (x-)}{T}\right)
        + \epsilon \,.
      \end{array}
    \right.
    \!\!\!
  \end{equation}
  Define now
  \begin{equation}
    \label{eq:5:V}
    \forall k\in \{1,\ldots,N\},\qquad
    y_k := \bar y + \frac{\eta}{N} (2k-1-N) \,.
  \end{equation}
  The intervals
  $\left] y_k - \frac{\eta}{N}, y_k+\frac{\eta}{N}\right]$ constitute
  a partition of $\left]\bar y-\eta, \bar y+\eta\right]$.

  Define $A_1, \ldots, A_N$ by
  \begin{equation}
    \label{eq:6:V}
    A_k(y):=
    \begin{cases}
      0, &\text{if }y\leq y_k-\frac{\eta}{N}
      \\
      \epsilon \left(1+ \frac{N}{\eta} (y-y_k)\right) , & \text{if
      }y_k-\frac{\eta}{N}\leq y\leq y_k
      \\
      \epsilon\left(1+ \frac{N}{\eta} (y_k-y)\right) , & \text{if }y_k
      \leq y\leq y_k+\frac{\eta}{N}
      \\
      0, & \text{if }y\leq y_k+\frac{\eta}{N}
    \end{cases}
  \end{equation}
  and then let $V_k := U_o - A_k$ for $k=1, \ldots, N$. Define
  $v_k := \partial_x V_k$ for $k = 1, \ldots, N$. It is easy to see
  using~\eqref{eq:32}, \eqref{eq:17} and~\eqref{eq:5:V} that
  $v_k\in \ii_T(w)$.  Finally, define
  $V_0:= U_o + \sum_{k=1}^N \left( U_o - V_k \right)$ and
  $v_0 := \partial_x V_o$. As above, $v_0\in \ii_T(w)$, since
  $V_k\geq U_o$ and $V_0=U_o$ outside $[\bar y-\eta,\bar
  y+\eta]$. Using the definition of $V_0$, condition~\eqref{eq:10} is
  trivially satisfied.

  Let us now consider scalars $\lambda_1, \ldots, \lambda_N$ such that
  \begin{equation}
    \label{eq:8:V}
    \sum_{k=1}^N \lambda_k \, \left(v_k-v_0\right) = 0 \,.
  \end{equation}
  Hence, $\sum_{k=1}^N \lambda_k \, \left(V_k-V_0\right) = C$ for a
  suitable $C \in \reali$. By the definitions of $V_0,\ldots,V_N$,
  $\forall k\in\{1,\ldots,N\}$, $V_k-V_0 = - A_k -\sum_{j=1}^N A_j$.
  So now~\eqref{eq:8:V} becomes
  $\sum_{k=1}^N \left( \lambda_k + \sum_{j=1}^N \lambda_j\right) A_k =
  -C$.  The $A_k$ have compact supports, so that $C=0$. Their supports
  are also disjoint, hence $\forall k \in \{1,\ldots,N\}$,
  $\lambda_k + \sum_{j=1}^N \lambda_j = 0$, from which it is clear
  that $\forall k\in \{1,\ldots,N\}$, $\lambda_k= 0$.
\end{proofof}

\medskip

\noindent\textbf{Acknowledgment:} Both authors thank Enrique Zuazua
for having suggested the problem and for useful discussions during the
\emph{``VII Partial Differential Equations, Optimal Design and
  Numerics''} that took place in Benasque. The second author was
supported by the ANR project \emph{``Finite4SOS''} (ANR~15-CE23-0007).

\medskip

{

  \small

  \bibliography{hcl_hj}

  \bibliographystyle{abbrv}

}

\bigskip
\end{document}